\documentclass[aos,preprint]{imsart} 

\RequirePackage[OT1]{fontenc}
\RequirePackage{amsthm,amsmath}
\RequirePackage[numbers]{natbib}
\RequirePackage[colorlinks,citecolor=blue,urlcolor=blue]{hyperref}

\usepackage{mathtools}
\usepackage{booktabs}
\usepackage[english]{babel} 
\usepackage{amsmath,amsfonts,amsthm}
\usepackage{ amssymb }
\usepackage{graphicx,psfrag,epsf}
\usepackage{array}
\usepackage{booktabs} 
\usepackage{natbib}
\usepackage{floatrow}
\usepackage{subfloat}
\usepackage{caption}
\usepackage{multirow}
\usepackage{epstopdf}
\usepackage{enumitem}
\usepackage{subfig}
\usepackage{color}
\usepackage{tabularx}
\usepackage[font = small,labelfont=bf,textfont=it]{caption} 
\usepackage{footnote}
\usepackage{algorithm}
\usepackage[noend]{algpseudocode}
\usepackage{rotating}

\makeatletter
\setlist[itemize]{leftmargin=*}
\newcommand{\distas}[1]{\mathbin{\overset{#1}{\kern\z@\sim}}}%
\newcommand{\bm}[1]{\mathbf{#1}}
\newsavebox{\mybox}\newsavebox{\mysim}
\newcommand{\distras}[1]{%
  \savebox{\mybox}{\hbox{\kern3pt$\scriptstyle#1$\kern3pt}}%
  \savebox{\mysim}{\hbox{$\sim$}}%
  \mathbin{\overset{#1}{\kern\z@\resizebox{\wd\mybox}{\ht\mysim}{$\sim$}}}%
}
\newtheorem{theorem}{Theorem}

\newtheorem{definition}{Definition}
\newtheorem{proposition}{Proposition}
\newtheorem{lemma}{Lemma}

\newtheorem{corollary}{Corollary}
\newcolumntype{C}[1]{>{\centering\let\newline\\\arraybackslash\hspace{0pt}}m{#1}}

\newcommand{\be}{\begin{equation}}
\newcommand{\ee}{\end{equation}}
\newcommand{\bi}{\begin{itemize}}
\newcommand{\ei}{\end{itemize}}
\newcommand{\ben}{\begin{enumerate}}
\newcommand{\een}{\end{enumerate}}
\newcommand{\stb}{\Statex $\bullet$ \;}
\newcommand{\crd}[1]{{\color{black}{#1}}}

\newfloatcommand{capbtabbox}{table}[][\FBwidth]
\allowdisplaybreaks
\DeclareMathOperator*{\argmin}{\arg\!\min}
\newcommand\smallO{
  \mathchoice
    {{\scriptstyle\mathcal{O}}}
    {{\scriptstyle\mathcal{O}}}
    {{\scriptscriptstyle\mathcal{O}}}
    {\scalebox{.7}{$\scriptscriptstyle\mathcal{O}$}}
  }
\makeatother

\let\oldbibliography\thebibliography
\renewcommand{\thebibliography}[1]{\oldbibliography{#1}
\setlength{\itemsep}{0pt}} 

\begin{document}

\begin{frontmatter}
\title{Support points}
\runtitle{Support points}

\begin{aug}
\author{\fnms{Simon} \snm{Mak} \ead[label=e1]{smak6@gatech.edu}} and
\author{\fnms{V. Roshan} \snm{Joseph} \ead[label=e2]{roshan@gatech.edu}}

\affiliation{Georgia Institute of Technology}
\runauthor{S. Mak and V. R. Joseph}

\end{aug}

\begin{abstract}
\; This paper introduces a new way to compact a continuous probability distribution $F$ into a set of representative points called support points. These points are obtained by minimizing the energy distance, a statistical potential measure initially proposed by Sz\'ekely and Rizzo (2004) for testing goodness-of-fit. The energy distance has two appealing features. First, its distance-based structure allows us to exploit the duality between powers of the Euclidean distance and its Fourier transform for theoretical analysis. Using this duality, we show that support points converge in distribution to $F$, and enjoy an improved error rate to Monte Carlo for integrating a large class of functions. Second, the minimization of the energy distance can be formulated as a difference-of-convex program, which we manipulate using two algorithms to efficiently generate representative point sets. In simulation studies, support points provide improved integration performance to both Monte Carlo and a specific Quasi-Monte Carlo method. Two important applications of support points are then highlighted: (a) as a way to quantify the propagation of uncertainty in expensive simulations, and (b) as a method to optimally compact Markov chain Monte Carlo (MCMC) samples in Bayesian computation.
\end{abstract}



\end{frontmatter}

\section{Introduction}
\label{sec:intro}

This paper explores a new method for compacting a continuous probability distribution $F$ into a set of representative points (rep-points) for $F$, which we call \textit{support points}. Support points have many important applications in a wide array of fields, because these point sets provide an improved representation of $F$ compared to a random sample. One such application is to the ``small-data'' problem of uncertainty propagation, where the use of support points as simulation inputs can allow engineers to quantify the propagation of input uncertainty onto system output at minimum cost. Another important application is to ``big-data'' problems encountered in Bayesian computation, specifically as a tool for compacting large posterior sample chains from Markov chain Monte Carlo (MCMC) methods \citep{Gey1992}. In this paper, we demonstrate the theoretical and practical effectiveness of support points for the general problem of integration, and illustrate its usefulness for the two applications above.

We first outline two classes of existing methods for rep-points. The first class consists of the so-called \textit{mse-rep-points} (see, e.g., Chapter 4 of \cite{FW1993}), which minimize the expected distance from a random point drawn from $F$ to its closest rep-point. Also known as principal points \citep{Flu1990}, mse-rep-points have been employed in a variety of statistical and engineering applications, including quantizer design \citep{GL2000, Pea2004} and optimal stratified sampling \citep{Dal1950, Cox1957}. In practice, these rep-points can be generated by first performing k-means clustering \citep{Llo1982} on a large batch sample from $F$, then taking the converged cluster centers as rep-points. One weakness of mse-rep-points, however, is that they do not necessarily converge to $F$ (see, e.g., \cite{Zad1982,Su2000}). The second class of rep-points, called \textit{energy rep-points}, aims to find a point set which minimizes some measure of statistical potential. Included here are the minimum-energy designs in \citep{Jea2015a} and the minimum Riesz energy points in \citep{Bea2014}. While the above point sets converge in distribution to $F$, its convergence rate is quite slow, both theoretically and in practice \citep{Bea2014}. Moreover, the construction of such point sets can be computationally expensive in high dimensions.
	
The key idea behind support points is that it {optimizes} a specific potential measure called the \textit{energy distance}, which makes such point sets a type of energy rep-point. First introduced in \cite{SR2004}, the energy distance was proposed as a computationally efficient way to evaluate goodness-of-fit (GOF), compared to the classical Kolmogorov-Smirnov (K-S) statistic \cite{Kol1933}, which is difficult to evaluate in high-dimensions. Similar to the existing energy rep-points above, we show in this paper that support points indeed converge in distribution to $F$. In addition, we demonstrate the improved error rate of support points over Monte Carlo for integrating a large class of functions. The minimization of this distance can also be formulated as a difference-of-convex (d.c.) program, which allows for efficient generation of support points.

Indeed, the \textit{reverse-engineering} of a GOF test forms the basis for state-of-the-art integration techniques called {Quasi-Monte Carlo} (QMC) methods (see \cite{DP2010} and \cite{Dea2013} for a modern overview). To see this, first let $g$ be a differentiable integrand, and let $\{\bm{x}_i\}_{i=1}^n$ be the point set (with empirical distribution, or e.d.f., $F_n$) used to approximate the desired integral $\int g(\bm{x}) \; dF(\bm{x})$ with the sample average $\int g(\bm{x}) \; dF_n(\bm{x})$. For simplicity, assume for now that $F = U[0,1]^p$ is the uniform distribution on the $p$-dimensional hypercube $[0,1]^p$, the typical setting for QMC. The Koksma-Hlawka inequality (see, e.g., \citep{Nie1992}) provides the following upper bound on the integration error $I$:
\small
\begin{equation}
I(g;F,F_n) \equiv \left| \int g(\bm{x}) \; d[F-F_n](\bm{x}) \right| \leq V_q(g) D_r(F, F_n), \; \;  1/q + 1/r = 1,
\label{eq:interr}
\end{equation}
\normalsize
where $V_q(g) = \|\partial^p g/\partial\bm{x}\|_{L_q}$, and $D_r(F, F_n)$ is the \textit{$L_r$-discrepancy}:
\begin{equation}
\small
D_r(F, F_n) =\left(\int \left|F_n(\bm{x}) - F(\bm{x})\right|^r\; d\bm{x} \right)^{1/r}.
\label{eq:disc}
\normalsize
\end{equation}
The discrepancy $D_r(F, F_n)$ measures how close the e.d.f. $F_n$ is to $F$, with a smaller value suggesting a better fit. Setting $r=\infty$, the $L_\infty$-discrepancy (or simply \textit{discrepancy}) becomes the classical K-S statistic for testing GOF. In other words, a point set with good fit to $F$ also provides reduced integration errors for a large class of integrands. A more general discussion of this connection in terms of kernel discrepancies can be found in \cite{Hic1999}.

For a general distribution $F$, the optimization of $D_r(F, F_n)$ can be a difficult problem. In the uniform setting $F = U[0,1]^p$, there has been some work on directly minimizing the discrepancy $D_\infty(F, F_n)$, including the cdf-rep-points in \citep{FW1993} and the uniform designs in \citep{Fan1980}. Such methods, however, are quite computationally expensive, and are applicable only for small point sets on $U[0,1]^p$ (see \citep{Fea2004a}). Because of this computational burden, modern QMC methods typically use {number-theoretic} techniques to generate point sets which achieve an \textit{asymptotically} quick decay rate for discrepancy. These include the {randomly-shifted lattice rules} \cite{Sea2002} using the component-by-component implementation of \cite{NC2006} (see also \cite{NK2014}), and the {randomly scrambled Sobol' sequences} due to \cite{Sob1967} and \cite{Owe1998}. While most QMC methods consider integration on the uniform hypercube $U[0,1]^p$, there are several ways to map point sets on $U[0,1]^p$ to non-uniform $F$. One such map is the inverse Rosenblatt transformation \citep{Ros1952}; however, it can be computed in closed-form only for a small class of distributions. When the density of $F$ is known up to a proportional constant, the Markov chain Quasi-Monte Carlo (MCQMC) approach \citep{OT2005} can also be used to generate QMC points on $F$.

\begin{figure}[t]
\centering
\includegraphics[width=0.9\linewidth]{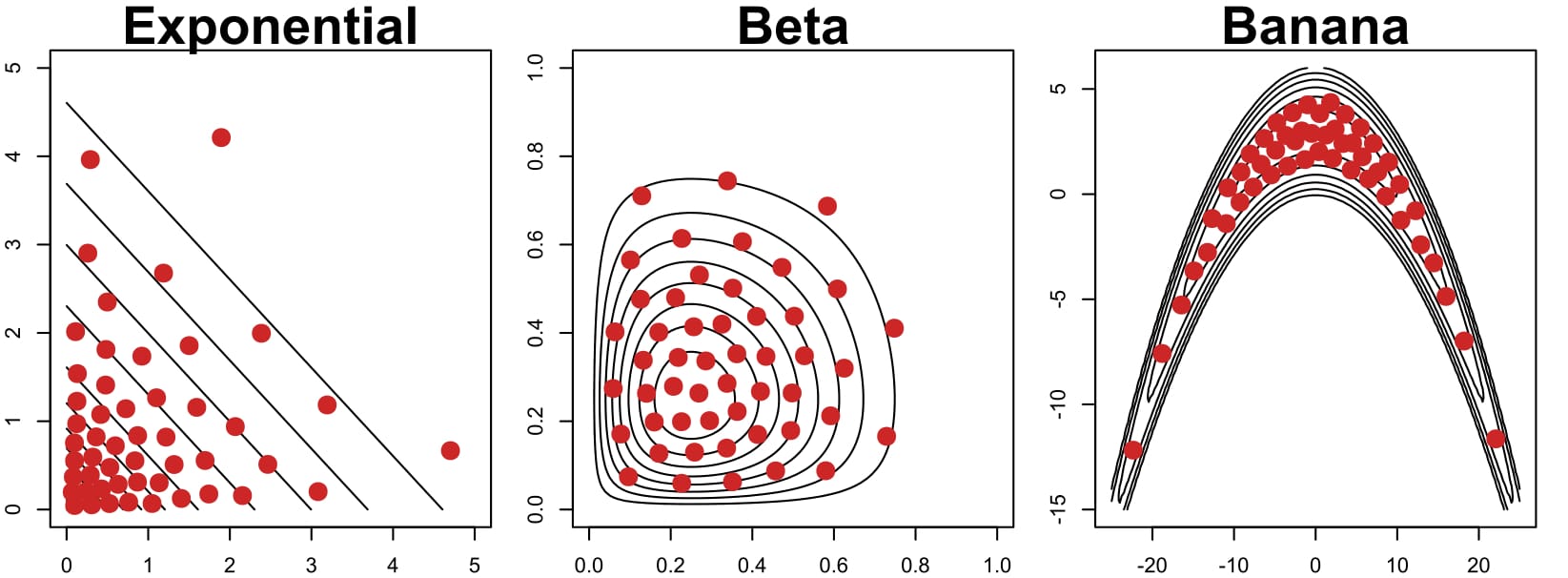}
\vspace{-0.3cm}
\caption{$n=50$ support points for 2-d i.i.d. $Exp(1)$, $Beta(2,4)$ and the banana-shaped distribution in \citep{GC2011}. Lines represent density contours.}
\label{fig:ex}
\end{figure}

Viewed in this light, the energy distance can be seen as a kernel discrepancy \citep{Hic1998} for non-uniform distributions, with the specific kernel choice being the negative Euclidean norm. However, in contrast with the typical number-theoretic construction of QMC point sets, support points are instead generated by optimizing the underlying d.c. formulation for the energy distance. This explicit optimization can have both advantages and disadvantages. On one hand, support points can be viewed as \textit{optimal} sampling points of $F$ (in the sense of minimum energy) for any desired sample size $n$. This optimality is evident in the three examples of support points plotted in Figure \ref{fig:ex} -- the points are concentrated in regions with high densities, but is sufficiently spread out to maximize the representativeness of each point. Such a ``space-filling'' property can allow for improved integration performance over existing QMC techniques, which we demonstrate in Section \ref{sec:sim}. On the other hand, the computational work for optimization can grow quickly when the desired sample size or dimension increases. To this end, we propose two algorithms which exploit the appealing d.c. formulation to efficiently generate point sets as large as 10,000 points in dimensions as large as 500.

This paper is organized as follows. Section \ref{sec:sp} proves several important theoretical properties of support points. Section \ref{sec:alg} proposes two algorithms for efficiently generating support points. Section \ref{sec:sim} outlines several simulations comparing the integration performance of support points with MC and an existing QMC method. Section \ref{sec:app} gives two important applications of support points in uncertainty propagation and Bayesian computation. Section \ref{sec:conc} concludes with directions for future research.

\section{Support points}
\label{sec:sp}

\subsection{Definition}
Let us first define the energy distance between two distributions $F$ and $G$:
\begin{definition}[Energy distance; Def. 1 of \citep{SZ2013}]
Let $F$ and $G$ be two distribution functions (d.f.s) on $\varnothing \neq \mathcal{X} \subseteq \mathbb{R}^p$ with finite means, and let $\bm{X}, \bm{X'} \distas{i.i.d.} G$ and $\bm{Y}, \bm{Y'} \distas{i.i.d.} F$. The \textup{energy distance} between $F$ and $G$ is defined as:
\begin{equation}
{E}(F,G) \equiv 2 \mathbb{E}\|\bm{X} - \bm{Y}\|_2 - \mathbb{E}\|\bm{X} - \bm{X'}\|_2 - \mathbb{E}\|\bm{Y} - \bm{Y'}\|_2.
\label{eq:mvl2}
\end{equation}
\end{definition}
\noindent When $G = F_n$ is the e.d.f. for $\{\bm{x}_i\}_{i=1}^n \subseteq \mathcal{X}$, this energy distance becomes:
\begin{equation}
{E}(F,F_n) = \frac{2}{n} \sum_{i=1}^n \mathbb{E}\|\bm{x}_i - \bm{Y}\|_2 - \frac{1}{n^2} \sum_{i=1}^n \sum_{j=1}^n \|\bm{x}_i - \bm{x}_j\|_2 - \mathbb{E}\|\bm{Y} - \bm{Y'}\|_2.
\label{eq:mvl2samp}
\end{equation}
\noindent For brevity, $F$ is assumed to be a continuous d.f. on $\varnothing \neq \mathcal{X} \subseteq \mathbb{R}^p$ with finite mean for the remainder of the paper.

The energy distance ${E}(F,F_n)$ was originally proposed in \citep{SR2004} as an efficient GOF test for high-dimensional data. In this light, support points are defined as the point set with best GOF under $E(F,F_n)$:

\begin{definition}[Support points]
Let $\bm{Y} \sim F$. For a fixed point set size $n \in \mathbb{N}$, the \textup{support points} of $F$ are defined as:
\begin{equation}
\small
\{\boldsymbol{\xi}_i\}_{i=1}^n \in \underset{\bm{x}_1, \cdots, \bm{x}_n}{\textup{Argmin}} \; {E}(F,F_n) = \underset{\bm{x}_1, \cdots, \bm{x}_n}{\textup{Argmin}} \left\{ \frac{2}{n} \sum_{i=1}^n \mathbb{E}\|\bm{x}_i - \bm{Y}\|_2 - \frac{1}{n^2} \sum_{i=1}^n \sum_{j=1}^n \|\bm{x}_i - \bm{x}_j\|_2 \right\}.
\tag{O}
\label{eq:spdef}
\end{equation}
\label{def:spdef}
\end{definition}
\vspace{-0.6cm}

The minimization of $E(F,F_n)$ is justified by the following {metric property}: 
\begin{theorem}[Energy distance, Prop. 2 of \citep{SZ2013}] ${E}(F,G) \geq 0$, with equality holding if and only if  $F {=} G$.
\end{theorem}
\noindent This theorem shows that the energy between two distributions is always non-negative, and equals zero if and only if these distributions are the same. In this sense, ${E}(F,G)$ can be viewed as a metric on the space of distribution functions. Support points, being the point set which minimizes such a metric, can then be interpreted as optimal sampling points which best represent $F$.

The choice of the energy distance $E(F,F_n)$ as an optimization objective is similar to its appeal in GOF testing. As mentioned in the Introduction, $E(F,F_n)$ was originally proposed as an efficient alternative to classical K-S statistic. However, not only is ${E}(F,F_n)$ {easy-to-evaluate}, it also has a desirable formulation as a d.c. program. We present in Section \ref{sec:alg} two algorithms which exploits this structure to efficiently generate support points.

In the univariate setting of $p = 1$, an interesting equivalence can be established between support points and optimal $L_2$-discrepancy points:
\begin{proposition}[Optimal $L_2$-discrepancy]
For a univariate d.f. $F$, the support points of $F$ are equal to the point set with minimal $L_2$-discrepancy.
\label{prop:l2disc}
\end{proposition}
\begin{proof}
It can be shown \citep{Sze2003} that ${E}(F,F_n) = 2 D_2^2(F,F_n)$, where $F_n$ is the e.d.f. of $\{{x}_i\}_{i=1}^n \subseteq \mathcal{X} \subseteq \mathbb{R}$ and $D_2(F, F_n)$ is the one-dimensional $L_2$-discrepancy in \eqref{eq:disc}. This proves the assertion.
\end{proof}

\noindent Unfortunately, such an equivalence fails to hold for $p>1$, since the $L_2$-discrepancy is not rotation-invariant. Support points and optimal $L_2$-discrepancy points can therefore behave quite differently in the multivariate setting.

\subsection{Theoretical properties}
\label{sec:theory}
While the notion of reverse engineering the energy distance is intuitively appealing, some theory is needed to demonstrate why the resulting points are appropriate for (a) representing the desired distribution $F$, and (b) integrating under $F$. To this end, we provide three theorems: the first proves the distributional convergence of support points to $F$, the second establishes a Koksma-Hlawka-like bound connecting integration error with ${E}(F,F_n)$, and the last provides an existence result for the resulting error convergence rate. The proofs of these results rely on the important property that, for generalized functions, the Fourier transform of the Euclidean norm $\|\cdot\|_2$ is proportional to the same norm raised to some power (see pg. 173-174 in \citep{GS1964}). We refer to various forms of this \textit{duality} property throughout the proofs.

\subsubsection{Convergence in distribution}
We first address the distributional convergence of support points to the desired distribution $F$:
\begin{theorem}[Distributional convergence]
Let $\bm{X} \sim F$ and $\bm{X}_n \sim F_n$, where $F_n$ is the e.d.f. of the support points in \eqref{eq:spdef}. Then $\bm{X}_n \xrightarrow{d} \bm{X}$.
\label{thm:consis}
\end{theorem}

\noindent This relies on the following lemma, which slightly extends the L\'evy continuity theorem to the {almost-everywhere (a.e.) pointwise convergence} setting.

\begin{lemma}
Let $(F_n)_{n=1}^\infty$ be a sequence of d.f.s with characteristic functions (c.f.s) $(\phi_n(\bm{t}))_{n=1}^\infty$, and let $F$ be a d.f. with c.f. $\phi(\bm{t})$. If $\bm{X}_n \sim F_n$ and $\bm{X} \sim F$, with $\lim_{n \rightarrow \infty} \phi_n(\bm{t}) = \phi(\bm{t})$ a.e. (in the Lebesgue sense), then $\bm{X}_n \xrightarrow{d} \bm{X}$.
\label{lem:dist}
\end{lemma}
\begin{proof}(Lemma \ref{lem:dist})
See Appendix A.1 of the supplemental article \citep{MJ2017supp}.
\end{proof}

\begin{proof}(Theorem \ref{thm:consis}) Define the sequence of random variables $(\bm{Y}_i)_{i=1}^\infty \distas{i.i.d.} F$, and let $\tilde{F}_n$ denote the e.d.f. of $\{\bm{Y}_i\}_{i=1}^n$. By the Glivenko-Cantelli lemma, $\lim_{n \rightarrow \infty} \sup_{\bm{x} \in \mathbb{R}^p} |\tilde{F}_n(\bm{x}) - F(\bm{x})| = 0$ a.s., so $\tilde{F}_n(\bm{x}) \rightarrow F(\bm{x})$ a.s. for all $\bm{x}$. Let $\phi(\bm{t})$ and $\tilde{\phi}_n(\bm{t})$ denote the c.f.s of $F$ and $\tilde{F}_n$, respectively. Since $|\exp(\text{i}\langle \bm{t}, \bm{x} \rangle)| \leq 1$, applying the Portmanteau theorem (Theorem 8.4.1 in \cite{Res1999}) and the dominated convergence theorem gives:
\begin{equation}
\lim_{n \rightarrow \infty} \mathbb{E}[ |\phi(\bm{t}) - \tilde{\phi}_n(\bm{t})|^2 ] = 0.
\label{eq:limexp}
\end{equation}

Using Prop. 1 of \cite{SZ2013} (this is a \textit{duality} result connecting the energy distance with c.f.s), the expected energy between $\tilde{F}_n$ and $F$ becomes:
\small
\begin{equation}
\mathbb{E} [{E}(F, \tilde{F}_n)] = \frac{1}{a_p} \mathbb{E} \left[ \int \frac{|\phi(\bm{t}) - \tilde{\phi}_n(\bm{t})|^2}{\|\bm{t}\|_2^{p+1}} \; d\bm{t} \right]  = \frac{1}{a_p} \int \frac{\mathbb{E} \left[|\phi(\bm{t}) - \tilde{\phi}_n(\bm{t})|^2 \right]}{\|\bm{t}\|_2^{p+1}} \; d\bm{t},
\label{eq:enaschar}
\end{equation}
\normalsize
where $a_p$ is some constant depending on $p$, with the last step following from Fubini's theorem. Note that $\mathbb{E} \left[|\phi(\bm{t}) - \tilde{\phi}_n(\bm{t})|^2 \right] = \frac{1}{n}\text{Var} \left[ \exp(\text{i}\langle \bm{t}, \bm{Y}_1 \rangle ) \right]$, so $\mathbb{E} \left[|\phi(\bm{t}) - \tilde{\phi}_n(\bm{t})|^2 \right]$ is monotonically decreasing in $n$. By the monotone convergence theorem and \eqref{eq:limexp}, we have:
\begin{equation}
\lim_{n \rightarrow \infty} \mathbb{E} [{E}(F, \tilde{F}_n)] = \frac{1}{a_p} \int \lim_{n \rightarrow \infty} \frac{ \mathbb{E} [|\phi(\bm{t}) - \tilde{\phi}_n(\bm{t})|^2]}{\|\bm{t}\|_2^{p+1}} \; d\bm{t} = 0.
\label{eq:pfconv}
\end{equation}

Consider now the e.d.f.s $(F_n)_{n=1}^\infty$ and c.f.s $(\phi_n)_{n=1}^\infty$ for support points. By Definition \ref{def:spdef}, ${E}(F, {F}_n) \leq \mathbb{E} [{E}(F, \tilde{F}_n) ]$, so $\lim_{n \rightarrow \infty} {E}(F, {F}_n) = 0$ by \eqref{eq:pfconv} and the squeeze theorem. Take any subsequence $(n_k)_{k=1}^\infty \subseteq \mathbb{N}_+$, and note that:
\[\lim_{k \rightarrow \infty} {E}(F, {F}_{n_k}) = \lim_{k \rightarrow \infty} \int \frac{|\phi(\bm{t}) - \phi_{n_k}(\bm{t})|^2}{\|\bm{t}\|_2^{p+1}} \; d\bm{t} = 0.\]
We know by the Riesz-Fischer Theorem (pg. 148 in \citep{RF2010}) that a sequence of functions $(f_n)$ which converge to $f$ in $L_2$ has a subsequence which converges pointwise a.e. to $f$. Applied here, this suggests the existence of a further subsequence $(n_k')_{k=1}^\infty \subseteq (n_k)_{k=1}^\infty$ satisfying ${\phi}_{n_k'}(\bm{t}) \stackrel{k \rightarrow \infty}{\rightarrow} \phi(\bm{t})$ a.e., so by Lemma \ref{lem:dist}, $\bm{X}_{n_k'} \xrightarrow{d} \bm{X}$. Since $(n_k)_{k=1}^\infty$ was arbitrarily chosen, it follows by the proof of Corollary 1 in Chapter 9 of \citep{Sho2000} that $\bm{X}_{n} \xrightarrow{d} \bm{X}$, which is as desired.
\end{proof}
In words, this theorem shows that support points are indeed representative of the desired distribution $F$ when the number of points $n$ grows large. From this, the \textit{consistency} of support points can be established:
\begin{corollary}[Consistency]
Let $\bm{X} \sim F$ and $\bm{X}_n \sim F_n$, with $F_n$ as in Theorem \ref{thm:consis}. (a) If $g: \mathcal{X} \rightarrow \mathbb{R}$ is continuous, then $g(\bm{X}_n) \xrightarrow{d} g(\bm{X})$. (b) If $g$ is continuous and bounded, then $\underset{n \rightarrow \infty}{\lim} \mathbb{E}[g(\bm{X}_n)] = \underset{n \rightarrow \infty}{\lim} \frac{1}{n} \sum_{i=1}^n g(\boldsymbol{\xi}_i) = \mathbb{E}[g(\bm{X})]$.
\label{corr:consis}
\end{corollary}
\begin{proof}
Part (a) follows from the continuous mapping theorem and Theorem \ref{thm:consis}. Part (b) follows by the Portmanteau theorem.
\end{proof}
\noindent The purpose of this corollary is two-fold: it demonstrates the {consistency} of support points for integration, and justifies the use of these point sets for a variety of other applications. Specifically, part (a) shows that support points are appropriate for performing uncertainty propagation in stochastic simulations, an application further explored in Section \ref{sec:integ}. Part (b) shows that any continuous and bounded integrand $g$ can be {consistently estimated} using support points, i.e., its sample average converges to the desired integral.

\subsubsection{A Koksma-Hlawka-like bound}
{Next, we present a theorem which upper bounds the squared integration error $I^2(g;F,F_n)$ by a term proportional to ${E}(F,F_n)$ for a large class of integrands. Such a result provides some justification on why the energy distance may be a \textit{good} criterion for integration. Here, we first provide a brief review of conditionally positive definite (c.p.d.) kernels, its native spaces, and their corresponding reproducing kernels, three ingredients which will be used for proving the desired theorem.\\

Consider the following definition of a \textit{conditionally positive definite kernel}:
\begin{definition}[c.p.d. kernel; Def. 8.1 of \citep{Wen2005}] A continuous function $\Phi: \mathbb{R}^p \rightarrow \mathbb{R}$ is a \textup{c.p.d. kernel} of order $m$ if, for all pairwise distinct $\bm{x}_1, \cdots, \bm{x}_N \in \mathbb{R}^p$ and all $\boldsymbol{\zeta} \in \mathbb{R}^N \setminus \{0\}$ satisfying $\sum_{j=1}^N \zeta_j p(\bm{x}_j) = 0$ for all polynomials of degree less than $m$, the quadratic form $\sum_{j=1}^N \sum_{k=1}^N \zeta_j \zeta_k \Phi(\bm{x}_j- \bm{x}_k)$ is positive.
\label{def:cpd}
\end{definition}

Similar to the theory of positive definite kernels (see, e.g., Section 10.1 and 10.2 of \cite{Wen2005}), one can use a c.p.d. kernel $\Phi$ to construct a reproducing kernel Hilbert space (RKHS) along with its reproducing kernel. This is achieved using the so-called \textit{native space} of $\Phi$:
\begin{definition}[Native space; Def. 10.16 of \citep{Wen2005}] Let $\Phi : \mathbb{R}^p \rightarrow \mathbb{R}$ be a c.p.d. kernel of order $m \geq 1$, and let $\mathcal{P}=\pi_{m-1}(\mathbb{R}^p)$ be the space of polynomials with degree less than $m$. Define the linear space:
\[\mathcal{F}_{\Phi}(\mathbb{R}^p) = \left\{ f(\cdot) = \sum_{j=1}^N \zeta_j \Phi(\bm{x}_j - \cdot) \; : \begin{array}{l}
\vspace{0.2cm}
N \in \mathbb{N}; \; \boldsymbol{\zeta} \in \mathbb{R}^N; \; \bm{x}_1, \cdots, \bm{x}_N \in \mathbb{R}^p,\\
\vspace{0.2cm}
\sum_{j=1}^N \zeta_j p(\bm{x}_j) = 0 \text{ for all } p \in \mathcal{P}
\end{array} \right\},\]
endowed with the inner product:
\[ \Bigg\langle \sum_{j=1}^N \zeta_j\Phi(\bm{x}_j- \cdot), \sum_{k=1}^M \zeta'_k \Phi(\bm{y}_k- \cdot) \Bigg\rangle_{\Phi}= \sum_{j=1}^N \sum_{k=1}^M \zeta_j \zeta'_k\Phi(\bm{x}_j - \bm{y}_k).\]
Let $\{\boldsymbol{\psi}_1, \cdots, \boldsymbol{\psi}_m\} \subseteq \mathbb{R}^p, m = \textup{dim}(\mathcal{P})$ be a $\mathcal{P}$-unisolvent subset\footnote{See Definition 2.6 of \cite{Wen2005}.}, and let $\{p_1, \cdots, p_m\} \subseteq \mathcal{P}$ be a Lagrange basis of $\mathcal{P}$ for such a subset. Furthermore, define the projective map $\Pi_\mathcal{P}:  C(\mathbb{R}^p)\footnote{$C(\mathbb{R}^p)$ is the space of continuous functions on $\mathbb{R}^p$.} \rightarrow \mathcal{P}$ as $\Pi_{\mathcal{P}}(f) = \sum_{k=1}^m f(\boldsymbol{\psi}_k) p_k$, and the map $\mathcal{R}: \mathcal{F}_{\Phi}(\mathbb{R}^p) \rightarrow C(\mathbb{R}^p)$ as $\mathcal{R}f(\bm{x}) = f(\bm{x}) - \Pi_{\mathcal{P}}f(\bm{x})$. The \textup{native space} for $\Phi$ is then defined as:
\[\mathcal{N}_{\Phi}(\mathbb{R}^p) = \mathcal{R}(\mathcal{F}_\Phi(\mathbb{R}^p)) + \mathcal{P},\]
and is equipped with the semi-inner product:
\[ \langle f, g \rangle_{\mathcal{N}_\Phi(\mathbb{R}^p)} = \langle \mathcal{R}^{-1}(f - \Pi_{\mathcal{P}} f), \mathcal{R}^{-1}(g - \Pi_{\mathcal{P}} g) \rangle_{\Phi}.\]
\label{def:native}
\end{definition}

\vspace{-0.6cm}

After obtaining the native space $\mathcal{N}_\Phi(\mathbb{R}^p)$, one can then define an appropriate inner product on $\mathcal{N}_\Phi(\mathbb{R}^p)$ to transform it into a RKHS:
\begin{theorem}[Native space to RKHS; Thm. 10.20 of \cite{Wen2005}] The native space $\mathcal{N}_\Phi(\mathbb{R}^p)$ for a c.p.d. kernel $\Phi$ carries the inner product $\langle f, g \rangle = \langle f, g\rangle_{\mathcal{N}_\Phi(\mathbb{R}^p)} + \sum_{k=1}^m f(\boldsymbol{\psi}_k) g(\boldsymbol{\psi}_k)$. With this inner product, $\mathcal{N}_\Phi(\mathbb{R}^p)$ becomes a reproducing kernel Hilbert space with reproducing kernel:
\begin{align*}
k(\bm{x},\bm{y}) &= \Phi(\bm{x}-\bm{y}) - \sum_{k=1}^m p_k(\bm{x})\Phi(\boldsymbol{\psi}_k-\bm{y}) - \sum_{l=1}^m p_l(\bm{y}) \Phi(\bm{x}-\boldsymbol{\psi}_l)\\
& \quad \quad + \sum_{k=1}^m\sum_{l=1}^m p_k(\bm{x}) p_l(\bm{y}) \Phi(\boldsymbol{\psi}_k-\boldsymbol{\psi}_l) + \sum_{k=1}^m p_k(\bm{x}) p_k(\bm{y}).
\end{align*}
\label{thm:rkhs}
\end{theorem}

\vspace{-0.4cm}

The following \textit{generalized Fourier transform} (GFT) will also be useful:
\begin{definition}[GFT; Defs. 8.8, 8.9 of \cite{Wen2005}] Suppose $f: \mathbb{R}^p \rightarrow \mathbb{C}$ is continuous and slowly increasing. A measurable function $\hat{f} \in \footnote{$L_2^{loc}$ denotes the space of locally $L_2$-integrable functions.}L_2^{loc}(\mathbb{R}^p \setminus \{0\})$ is called the \textup{generalized Fourier transform} of $f$ if $\exists m \in \mathbb{N}_0/2$ such that $\int_{\mathbb{R}^p} f(\bm{x}) \hat{\gamma}(\bm{x}) \; d\bm{x} = \int_{\mathbb{R}^p} \hat{f}(\omega) \gamma(\omega) \; d\omega$ is satisfied for all $\gamma \in \mathcal{S}_{2m}$, where $\hat{\gamma}$ denotes the standard Fourier transform of $\gamma$. Here, $\mathcal{S}_{2m} = \{ \gamma \in \mathcal{S} \; : \; \gamma(\omega) = \mathcal{O}(\|\omega\|_2^{2m}) \text{ for } \|\omega\|_2 \rightarrow 0\}$, where $\mathcal{S}$ is the Schwartz space.
\label{def:gft}
\end{definition}
\noindent Specific definitions for slowly increasing functions and Schwartz spaces can be found in Definitions 5.19 and 5.17 of \cite{Wen2005}. Here, the order of the GFT $\hat{f}$ refers to the value $m$ in Definition \ref{def:gft}, which can reside on the half-integers $\mathbb{N}_0/2$ since the index of the underlying space $\mathcal{S}_{2m}$ will still be an integer. \\

With these concepts in hand, we now present the Koksma-Hlawka-like bound. As demonstrated below, the choice of the negative distance kernel $\Phi = -\| \cdot \|_2$ is important for connecting integration error with the distance-based energy distance $E(F,F_n)$.
\begin{theorem}[Koksma-Hlawka]
Let $\{\bm{x}_i\}_{i=1}^n \subseteq \mathcal{X} \subseteq \mathbb{R}^p$ be a point set with e.d.f. $F_n$, and let $\Phi(\bm{x}) = -\| \bm{x} \|_2$. Then $\Phi$ is a c.p.d. kernel of order 1. Moreover:
\begin{enumerate}[label=(\alph*)]
\item The native space of $\Phi$, $\mathcal{N}_\Phi(\mathbb{R}^p)$, can be explicitly written as:
\begin{equation}
\hspace*{-0.8cm}
\mathcal{N}_\Phi(\mathbb{R}^p) = \left\{ f \in C(\mathbb{R}^p) : \begin{array}{l}
\vspace{0.2cm}
\textup{(G1) } \exists m \in \mathbb{N}_0 \text{ s.t. } f(\bm{x}) = \mathcal{O}(\|\bm{x}\|_2^m) \text{ for } \|\bm{x}\|_2 \rightarrow \infty\\
\textup{(G2) } f \text{ has a GFT } \hat{f} \text{ of order } 1/2\\
\textup{(G3) } \int \|\omega\|_2^{p+1} |\hat{f}(\omega)|^2 \; d\omega < \infty
\end{array} \right\},
\label{eq:g}
\end{equation}
with semi-inner product given by:
\begin{equation}
\langle f,g \rangle_{\mathcal{N}_\Phi(\mathbb{R}^p)} = \left\{ \Gamma((p+1)/2) 2^{p} \pi^{(p-1)/2} \right\}^{-1} \int \hat{f}(\omega) \overline{\hat{g}(\omega)} \|\omega\|_2^{p+1} \; d \omega,
\label{eq:inner}
\end{equation}
\item Consider the function space $\mathcal{G}_p = \mathcal{N}_\Phi(\mathbb{R}^p)$, equipped with inner product $\langle f,g \rangle_{\mathcal{G}_p} = \langle f,g \rangle_{\mathcal{N}_\Phi(\mathbb{R}^p)} + f(\boldsymbol{\psi})g(\boldsymbol{\psi})$ for a fixed choice of $\boldsymbol{\psi} \in \mathcal{X}$. Then $(\mathcal{G}_p, \langle \cdot, \cdot \rangle_{\mathcal{G}_p})$ is a RKHS, and for any integrand $g \in \mathcal{G}_p$, the integration error in \eqref{eq:interr} is bounded by:
\begin{equation}
I(g;F,F_n) \leq \|g\|_{\mathcal{G}_p} \sqrt{{E}(F,F_n)}, \quad \|g\|^2_{\mathcal{G}_p} \equiv \langle g,g \rangle_{\mathcal{G}_p}.
\label{eq:upbound}
\end{equation}
\end{enumerate}
\label{thm:interr}
\end{theorem}
}

\begin{proof}
(Theorem \ref{thm:interr}) Consider first part (a). Let $\Phi(\cdot) = -\|\cdot\|_2$, and let $\hat{\Phi}$ be its GFT of order 1. {From Theorem 8.16 of \citep{Wen2005}, we have the following \textit{duality} representation:}
\[\hat{\Phi}(\omega) = \frac{2^{p/2}\Gamma(({p+1})/{2})}{\sqrt{\pi}}\|\omega\|_2^{-p-1}, \quad \omega \in \mathbb{R}^p \setminus \{0\}.\]
By Corollary 8.18 of \citep{Wen2005}, $\Phi(\cdot)$ is also c.p.d. of order 1. Using the fact that $\Phi(\cdot)$ is even along with the continuity of $\hat{\Phi}(\omega)$ on $\mathbb{R}^p \setminus \{0\}$, an application of Theorem 10.21 in \citep{Wen2005} completes the proof for part (a).\\

Consider now part (b). By Prop. 3 of \citep{SZ2013}, the kernel $\Phi(\cdot)$ is c.p.d. with respect to the space of constant functions $\mathcal{P} = \{f(\bm{x}) \equiv C \text{ for some } C \in \mathbb{R}\}$, with $\dim \mathcal{P} = 1$. Note that any choice of $\boldsymbol{\psi} \in \mathcal{X}$ provides a $\mathcal{P}$-unisolvent subset, with the Lagrange basis for the single point $\boldsymbol{\psi}$ being the unit function $p(\cdot) \equiv 1$. Hence, by Theorem \ref{thm:rkhs}, the native space $\mathcal{N}_{\Phi}(\mathbb{R}^p)$ can be transformed into a RKHS $\mathcal{G}_p$ by equipping it with a new inner product $\langle f,g \rangle_{\mathcal{G}_p} = \langle f,g \rangle_{\mathcal{N}_{\Phi}(\mathbb{R}^p)} + f(\boldsymbol{\psi})g(\boldsymbol{\psi})$. From the same theorem, the corresponding reproducing kernel for the RKHS $(\mathcal{G}_p, \langle \cdot, \cdot \rangle_{\mathcal{G}_p})$ becomes $\tilde{k}(\bm{x},\bm{y}) = \Phi(\bm{x}-\bm{y}) - \Phi(\boldsymbol{\psi}-\bm{y}) - \Phi(\boldsymbol{\psi}-\bm{x}) + 1$.\\

{Next, let $\tilde{k}_{\bm{x}}(\bm{z}) = \tilde{k}(\bm{x},\bm{z})$. We claim the function $\int \tilde{k}_{\bm{x}}(\cdot) \; d[F-F_n](\bm{x})$ belongs in $\mathcal{G}_p$. To see this, define the linear operator $\mathcal{L}: \mathcal{G}_p \rightarrow \mathbb{R}$ as $\mathcal{L}f = \int f(\bm{x}) \; dF(\bm{x})$. Note that $\mathcal{L}$ is a bounded operator, because for all $f \in \mathcal{G}_p$:
\begin{align*}
|\mathcal{L}f| = \Big|\int f(\bm{x}) \; dF(\bm{x}) \Big| &\leq \int |f(\bm{x})| \; dF(\bm{x})\\
&= \int | \langle f(\cdot), \tilde{k}_{\bm{x}}(\cdot) \rangle_{\mathcal{G}_p} | \; dF(\bm{x}) \quad \text{(RKHS reproducing property)}\\
&\leq \int \|f\|_{\mathcal{G}_p} \| \tilde{k}_{\bm{x}}(\cdot) \|_{\mathcal{G}_p} \; dF(\bm{x}) \quad \text{(Cauchy-Schwarz)}\\
&= \|f\|_{\mathcal{G}_p} \int \tilde{k}^{1/2}(\bm{x},\bm{x}) \; dF(\bm{x}), \quad \text{(RKHS kernel trick)}
\end{align*}
and the last expression must be bounded because $\int \tilde{k}^{1/2}(\bm{x},\bm{x}) \; dF(\bm{x}) \leq [\int \tilde{k}(\bm{x},\bm{x}) \; dF(\bm{x})]^{1/2}$, the latter of which is finite due to the earlier finite mean assumption on $F$. By the Riesz Representation Theorem (Theorem 8.12, \cite{HN2001}), there exists a unique $\tilde{f} \in \mathcal{G}_p$ satisfying $\mathcal{L}f = \int f(\bm{x}) \; dF(\bm{x}) = \langle f, \tilde{f} \rangle_{\mathcal{G}_p}$ for all $f \in \mathcal{G}_p$. Setting $f(\bm{x}) = \tilde{k}_{\bm{z}}(\bm{x})$ in this expression, we get $\int \tilde{k}_{\bm{z}}(\bm{x}) \; dF(\bm{x}) = \langle \tilde{k}_{\bm{z}}(\cdot), \tilde{f} \rangle_{\mathcal{G}_p} = \tilde{f}(\bm{z})$ by the RKHS reproducing property, so $\tilde{f} = \int \tilde{k}_{\bm{x}}(\cdot)\; dF(\bm{x}) \in \mathcal{G}_p$. Finally, note that $\int \tilde{k}_{\bm{x}}(\cdot) \; dF_n(\bm{x}) \in \mathcal{G}_p$ because a RKHS is closed under addition, so $\int \tilde{k}_{\bm{x}}(\cdot) \; d[F-F_n](\bm{x}) \in \mathcal{G}_p$, as desired.}\\

With this in hand, the integration error can be bounded as follows:
\begin{align*}
I(g;F,F_n) &= \left| \int g(\bm{x}) \; d[F-F_n](\bm{x}) \right| &\\
& = \left| \int \Big\langle g(\cdot), \tilde{k}_{\bm{x}}(\cdot) \Big\rangle_{\mathcal{G}_p} \; d[F-F_n](\bm{x}) \right| & \text{(Reproducing property)}\\
& = \left| \Big\langle g(\cdot), \int \tilde{k}_{\bm{x}}(\cdot) \; d[F-F_n](\bm{x}) \Big\rangle_{\mathcal{G}_p} \right| & \\
& \leq \| g \|_{\mathcal{G}_p} \left\| \int \tilde{k}_{\bm{x}}(\cdot) \; d[F-F_n](\bm{x}) \right\|_{\mathcal{G}_p}.& \text{(Cauchy-Schwarz)}
\end{align*}
The last term can be rewritten as:
\small
\begin{align*}
\sqrt{\left\| \int \tilde{k}_{\bm{x}}(\cdot)\; d[F-F_n](\bm{x}) \right\|^2_{\mathcal{G}_p}} &= \sqrt{ \Big\langle \int \tilde{k}_{\bm{x}}(\cdot)\; d[F-F_n](\bm{x}), \int \tilde{k}_{\bm{y}}(\cdot) \; d[F-F_n](\bm{y}) \Big\rangle_{\mathcal{G}_p}} \\
\normalsize
& = \sqrt{ \int \int \langle \tilde{k}_{\bm{x}}(\cdot), \tilde{k}_{\bm{y}}(\cdot) \rangle_{\mathcal{G}_p} \; d[F-F_n](\bm{x}) \; d[F-F_n](\bm{y}) }\\
& = \sqrt{ \int \int \tilde{k}(\bm{x},\bm{y}) \; d[F-F_n](\bm{x}) \; d[F-F_n](\bm{y}) }\\
& \hspace{6cm} \text{(Kernel trick)}\\
& = \sqrt{ \int \int \Phi(\bm{x}-\bm{y}) \; d[F-F_n](\bm{x}) \; d[F-F_n](\bm{y}) }\\
&= \sqrt{E(F,F_n)}, \hspace{3.5cm} \text{(Equation \eqref{eq:mvl2samp})}
\end{align*}
\normalsize
where the second-last step follows because $\int \Phi(\boldsymbol{\psi}-\bm{y}) \; d[F-F_n](\bm{x}) = \int \Phi(\boldsymbol{\psi}-\bm{x}) \; d[F-F_n](\bm{y}) = \int d[F-F_n](\bm{x}) = 0$. This completes the proof.
\end{proof}

The appeal of Theorem \ref{thm:interr} is that it connects the integration error $I(g;F,F_n)$ with the energy distance ${E}(F,F_n)$ for all integrands $g$ in the function space $\mathcal{G}_p$. Similar to the usual Koksma-Hlawka inequality, such a theorem justifies the use of support points for integration, because the integration error for all functions in $\mathcal{G}_p$ can be sufficiently bounded by minimizing ${E}(F,F_n)$.

{A natural question to ask is how large $\mathcal{G}_p$ is compared with the commonly-used Sobolev space $W_{s,2}$, i.e., the set of functions whose $s$-th order differentials have finite $L_2$ norm. Such a comparison is particularly important in light of the fact that an anchored variant of the Sobolev space is typically employed in QMC analysis (see, e.g., \citep{Dea2013}). Recall that $s$ can be extended to the non-negative real numbers using fractional calculus, in which case $W_{s,2}$ becomes the fractional Sobolev space. By comparing the definition of the fractional Sobolev space in the Fourier domain (see (3.7) in \cite{Nea2012}), one can show that $W_{(p+1)/2,2}$ is contained within $\mathcal{G}_p$. Moreover, using the fact that $W_{s,2}$ is a decreasing family as $s>0$ increases (see paragraph prior to Prop. 1.52 in \citep{Bea2011}), it follows that $W_{\lceil (p+1)/2 \rceil,2} \subseteq W_{(p+1)/2,2} \subseteq \mathcal{G}_p$. In fact, for odd dimensions $p$, Theorem 10.43 of \citep{Wen2005} shows that $\mathcal{G}_p$ is indeed equal to the Sobolev space $W_{\lceil (p+1)/2 \rceil,2} = W_{(p+1)/2,2}$, so the embedding result becomes an equality.

Viewing this embedding now in terms of Theorem \ref{thm:interr}, it follows that all integrands $g$ with square-integrable $\lceil (p+1)/2 \rceil$-th order differentials enjoy the upper bound in \eqref{eq:upbound}. Hence, as dimension $p$ grows, an increasing order of smoothness is required for integration using support points, which appears to be a necessary trade-off for the appealing d.c. formulation in \eqref{eq:spdef}. This is similar to the anchored Sobolev spaces employed in QMC, which requires integrands to have square-integrable mixed first derivatives.}

\subsubsection{Error convergence rate}
Next, we investigate the convergence rate of $I(g;F,F_n)$ under support points. Under eigenvalue decay conditions, the following theorem establishes an \textit{existence} result, which demonstrates the existence of a point set sequence achieving a particular error rate. An additional theorem then clarifies when such decay conditions are satisfied in practice. The main purpose of these results is to demonstrate the quicker \textit{theoretical} convergence of support points over Monte Carlo. From the simulations in Section \ref{sec:sim}, the rate below does not appear to be tight, and a quicker convergence rate is conjectured in Appendix A.3 of the supplemental article \citep{MJ2017supp}.

\begin{theorem}[Error rate]
Let $F_n$ be the e.d.f. for support points $\{\boldsymbol{\xi}\}_{i=1}^n$, and let $g \in \mathcal{G}_p$. Define the kernel $k(\bm{x},\bm{y}) = \mathbb{E}\|\bm{x}-\bm{Y}\|_2 + \mathbb{E}\|\bm{y}-\bm{Y}\|_2 - \mathbb{E}\|\bm{Y}-\bm{Y}'\|_2 - \|\bm{x}-\bm{y}\|_2$, $\bm{Y},\bm{Y}' \distas{i.i.d.} F$. If (a) $\mathbb{E}[\|\bm{Y}\|_2^3] < \infty$, and (b) the weighted eigenvalues of $k$ under $F$ satisfy $\sum_{k=1}^\infty \lambda_k^{1/\alpha} < \infty$ for some $\alpha > 1$, then:
\begin{equation}
I(g;F,F_n) = \mathcal{O}\{\|g\|_{\mathcal{G}_p} {n}^{-1/2}(\log n)^{-(\alpha-1)/2}\},
\label{eq:intconvrate}
\end{equation}
with constant terms depending on $\alpha$ and $p$.
\label{thm:convrate}
\end{theorem}
\noindent Here, the weighted eigenvalue sequence of $k$ under $F$ is the decreasing sequence $(\lambda_k)_{k=1}^\infty$ satisfying $\lambda_k \phi_k(\bm{x}) = \mathbb{E} [k(\bm{x},\bm{Y}) \phi_k(\bm{Y})]$, $\mathbb{E}[\phi_k^2(\bm{Y})] = 1$.

The proof of this theorem exploits the fact that $E(F,F_n)$ is a goodness-of-fit statistic. Specifically, writing ${E}(F,F_n)$ as a {degenerate V-statistic} $V_n$, we appeal to its limiting distribution and a uniform Barry-Esseen-like rate to derive an upper bound for the minimum of $V_n$. The full proof is outlined below, and relies on the following lemmas.

\begin{lemma}
\textup{(\citep{Ser2009})} Let $(\bm{Y}_i)_{i=1}^\infty \distas{i.i.d.} F$, and let $k$ be a symmetric, positive definite (p.d.) kernel with $\mathbb{E}[k(\bm{x},\bm{Y}_1)] = 0$, $\mathbb{E}[k^2(\bm{Y}_1, \bm{Y}_2)] < \infty$ and $\mathbb{E}|k(\bm{Y}_1, \bm{Y}_1)| < \infty$. Define the V-statistic $V_n \equiv n^{-2}\sum_{i=1}^n \sum_{j=1}^n k(\bm{Y}_i,\bm{Y}_j)$. Then $W_n \equiv n V_n \xrightarrow{d} \sum_{k=1}^\infty \lambda_k \chi_{k}^2 \equiv W_{\infty}$, where $(\chi_{k}^2)_{k=1}^\infty \distas{i.i.d.} \chi^2(1)$, and $(\lambda_k)_{k=1}^\infty$ are the weighted eigenvalues of $k$ under $F$.
\label{lem:asymp}
\end{lemma}

\begin{lemma}
\textup{(\citep{KB1989})} Adopt the same notation as in Lemma \ref{lem:asymp}, and let $F_{W_n}$ and $F_{W_\infty}$ denote the d.f.s for $W_n$ and $W_\infty$. If $\mathbb{E}[k(\bm{x},\bm{Y}_1)] = 0$, $\mathbb{E}|k(\bm{Y}_1, \bm{Y}_2)|^3 < \infty$ and $\mathbb{E}|k(\bm{Y}_1, \bm{Y}_1)|^{3/2} < \infty$, then:
\begin{equation}
\sup_{x} |F_{W_n}(x) - F_{W_\infty}(x)| = \smallO(n^{-1/2}),
\label{eq:conc}
\end{equation}
\label{lem:conc}
with constants depending on dimension $p$.
\end{lemma}

\begin{lemma}[Paley-Zygmund inequality; \citep{PZ1930}] Let $X \geq 0$, with constants $a_1>1$ and $a_2>0$ satisfying $\mathbb{E}(X^2) \leq a_1 \mathbb{E}^2(X)$ and $\mathbb{E}(X) \geq a_2$. Then, for any $\theta \in (0,1)$, $\mathbb{P}(X \geq a_2 \theta) \geq (1-\theta)^2/a_1$.
\label{lem:pz}
\end{lemma}

\noindent The proof of Theorem \ref{thm:convrate} then follows:

\begin{proof} (Theorem \ref{thm:convrate}) Following Section 7.4 of \cite{SZ2013}, the energy distance ${E}(F,F_n)$ can be written as the order-2 $V$-statistic:
\begin{equation}
{E}(F,F_n) = \frac{1}{n^2}\sum_{i=1}^n \sum_{j=1}^n k(\boldsymbol{\xi}_i,\boldsymbol{\xi}_j) \leq \frac{1}{n^2}\sum_{i=1}^n \sum_{j=1}^n k(\bm{Y}_i,\bm{Y}_j) \equiv V_n,
\end{equation}
where $k(\bm{x},\bm{y})$ is defined in Theorem \ref{thm:convrate} and $(\bm{Y}_i)_{i=1}^n \distas{i.i.d.} F$. The last inequality follows by the definition of support points.

By \citep{Yan2012}, the kernel $k$ is symmetric and p.d., and the conditions for Lemma \ref{lem:asymp} can easily be shown to be satisfied. Invoking this lemma, we have:
\begin{equation}
\inf\{x: F_{W_n}(x) > 0\} = n {E}(F,F_n),
\label{eq:equiv}
\end{equation}
The strategy is to lower bound the left-tail probability of $W_{\infty}$, then use this to derive an upper bound for $\inf\{x: F_{W_n}(x) > 0\}$ using Lemma \ref{lem:conc}.

We first investigate the left-tail behavior of $W_{\infty}$. Define $Z_t = \exp\{-tW_{\infty}\}$ for some $t > 0$ to be determined later. Since $Z_t$ is bounded a.s., $\mathbb{E}(Z_t) = \prod_{k=1}^\infty (1+2\lambda_k t)^{-1/2}$ and $\mathbb{E}(Z_t^2) = \prod_{k=1}^\infty (1+4\lambda_k t)^{-1/2}$. From Lemma \ref{lem:pz}, it follows that, for fixed $x >0$, if our choice of $t$ satisfies:
\begin{equation}
\textbf{[A1]}: \;\mathbb{E}(Z_t) \geq 2\exp\{-tx\} >  \exp\{-tx\}, \;  \textbf{[A2]}: \; \mathbb{E}(Z_t^2) \leq a_1\mathbb{E}^2(Z_t),
\label{eq:conds}
\end{equation}
then, setting $\theta = 1/2$ and $a_2 =  2\exp\{-tx\}$, we have:
\begin{equation}
F_{W_{\infty}}(x) = \mathbb{P}(Z_t \geq \exp\{-tx\}) \geq \mathbb{P}(Z_t \geq \mathbb{E}(Z_t)/2) \geq (4a_1)^{-1}.
\label{eq:cdfw}
\end{equation}

Consider \textbf{[A1]}, or equivalently: $tx \geq \log 2 + (1/2)\sum_{k=1}^\infty \log (1+2 \lambda_k t)$. Since $\log (1+x) \leq x^q/q$ for $x > 0$ and $0 < q < 1$, and $\sum_{k=1}^\infty \lambda_k^{1/\alpha} < \infty$ by assumption, a sufficient condition for \textbf{[A1]} is:
\begin{align*}
tx \geq \log 2 + ({\alpha}/{2})\sum_{k=1}^\infty (2 \lambda_k t)^{1/\alpha} & \Leftrightarrow \; P_\alpha(s) \equiv s^\alpha - b_p s x^{-1} - (\log 2)x^{-1} \geq 0,
\end{align*}
where $s = t^{1/\alpha}$ and $b_p = \alpha 2^{1/\alpha-1}\sum_{k=1}^\infty \lambda_k^{1/\alpha}>0$.

Since $\log 2 > 0$ and $b_p s x^{-1} > 0$, there exists exactly one (real) positive root for $P_\alpha(s)$. Call this root $r$, so the above inequality is satisfied for $s > r$. Define $\bar{P}_\alpha(s)$ as the linearization of $P_\alpha(s)$ for $s > \bar{s}=(b_px^{-1})^{1/(\alpha-1)}$, i.e.:
\[
\bar{P}_\alpha(s) = \begin{cases}	
P_\alpha(s), & 0 \leq s \leq \bar{s}\\
-x^{-1}\log 2 +P_\alpha'(\bar{s}) \cdot (s - \bar{s}), & s > \bar{s}.
\end{cases}
\]

From this, the unique root of $\bar{P}_\alpha(s)$ can be shown to be $\bar{r} = \bar{s} + x^{-1}(\log 2)[P_\alpha'(\bar{s})]^{-1}$. Since $P_\alpha(s) \geq \bar{P}_\alpha(s)$ for all $s \geq 0$, $\bar{r} \geq r$, the following upper bound for $\bar{r}$ can be obtained for sufficiently small $x$:
\[\bar{r} = (b_px^{-1})^{1/(\alpha-1)} + (\log 2)(\alpha - 1)^{-1} b_p^{-1} \leq 2 (b_p x^{-1})^{1/(\alpha-1)}.\]
Hence:
\begin{align}
\begin{split}
t = s^\alpha \geq 2^\alpha (b_px^{-1})^{\alpha/(\alpha-1)} &\Leftrightarrow s \geq 2 (b_px^{-1})^{1/(\alpha-1)} \; \geq \bar{r} \geq r \\
& \Rightarrow \; s^\alpha - b_px^{-1}s - (\log 2) x^{-1} \geq 0,
\label{eq:cond1}
\end{split}
\end{align}
so setting $t = 2^\alpha (b_px^{-1})^{\alpha/(\alpha-1)} \equiv c_px^{-\alpha/(\alpha-1)}$ satisfies \textbf{[A1]} in \eqref{eq:conds}. 


The next step is to determine the smallest $a_1$ satisfying \textbf{[A2]} in \eqref{eq:conds}, or equivalently, $\frac{1}{2}\sum_{k=1}^\infty \log(1+4\lambda_k t) \geq \sum_{k=1}^\infty \log(1+2\lambda_k t) - \log a_1$. Again, since $\log (1+x) \leq x^q/q$ for $x > 0$ and $0 < q < 1$, a sufficient condition for \textbf{[A2]} is:
\begin{align*}
\log a_1 \geq \sum_{k=1}^\infty \log(1+2\lambda_k t) \Leftarrow \log a_1 \geq \alpha\sum_{k=1}^\infty (2\lambda_k t)^{1/\alpha}
\end{align*}
Plugging in $t =c_px^{-\alpha/(\alpha-1)}$ from \eqref{eq:cond1} and letting $d_{p}\equiv \alpha(2c_{p})^{1/\alpha}\left(\sum_{k=1}^\infty \lambda_k^{1/\alpha}\right)$, we get $\log a_1 \geq d_{p} x^{-1/(\alpha-1)} \Leftrightarrow a_1 \geq \exp\left\{ d_{p} x^{-1/(\alpha-1)} \right\}$.

The choice of $t =c_{p}x^{-\alpha/(\alpha-1)}$ and $a_1 = \exp\left\{ d_{p} x^{-1/(\alpha-1)} \right\}$ therefore \textbf{[A1]} and \textbf{[A2]} in \eqref{eq:cdfw}. It follows from \eqref{eq:cdfw} that:
\begin{equation}
F_{W_{\infty}}(x) \geq (4a_1)^{-1} = \exp\{ - d_{p} x^{-1/(\alpha-1)} \}/4,
\label{eq:cdfrate}
\end{equation}
so $F_{W_{\infty}}(x)$ converges to 0 at a rate of $\mathcal{O}(\exp\left\{ - d_{p} x^{-1/(\alpha-1)} \right\})$ as $x \rightarrow 0^+$.

Consider now the behavior of $\inf\{x: F_{W_n}(x) > 0\}$ as $n \rightarrow \infty$. From the uniform bound in Lemma \ref{lem:conc}, there exists a sequence $(c_{n,p})_{n=1}^\infty, \lim_{n \rightarrow \infty} c_{n,p} = 0$ such that $|F_{W_n}(x) - F_{W_{\infty}}(x)| \leq c_{n,p} n^{-1/2}$ for all $x \geq 0$. Setting the right side of \eqref{eq:cdfrate} equal to $2c_{n,p}n^{-1/2}$ and solving for $x$, we get:
\begin{equation}
x^* = \frac{d_{p}^{\alpha-1}}{[\frac{1}{2}\log n - \log (8c_{n,p})]^{\alpha-1}} \Rightarrow F_{W_{\infty}}(x^*) \geq \exp\left\{- d_{p} (x^*)^{-1/(\alpha-1)} \right\} = 2c_{n,p} n^{-1/2}.
\label{eq:equal}
\end{equation}
so Lemma \ref{lem:conc} ensures the above choice of $x^*$ satisfies $F_{W_n}(x^*) \geq c_{n,p} n^{-1/2} > 0$.

Using this with \eqref{eq:equiv}, it follows that:
\begin{align*}
{E}(F,F_n) = \mathcal{O}\left\{n^{-1}(\log n)^{-(\alpha-1)}\right\},
\end{align*}
with constants depending on $p$. Finally, by Theorem \ref{thm:interr}, we have:
\[I(g;F,F_n) = \mathcal{O}\{\|g\|_{\mathcal{G}_p} {n}^{-1/2}(\log n)^{-(\alpha-1)/2}\}\]
which is as desired.
\end{proof}

The following theorem provides some insight on when the eigenvalue decay condition $\sum_{k=1}^\infty \lambda_k^{1/\alpha} < \infty$ in Theorem \ref{thm:convrate} is satisfied.

\begin{theorem}[Eigenvalue conditions]
Let $F_n$ and $F$ be as in Theorem \ref{thm:convrate}, and let $g \in \mathcal{G}_p$.
\begin{enumerate}[label=(\alph*)]
\item If $\mathcal{X} \subseteq \mathbb{R}^p$ is a bounded Borel set with non-empty interior, then $I(g;F,F_n) = \mathcal{O}\{\|g\|_{\mathcal{G}_p} {n}^{-1/2}(\log n)^{-(1-\nu)/(2p)}\}$ for any $\nu \in (0,1)$,
\item If $\mathcal{X} \subseteq \mathbb{R}^p$ is measurable with positive Lebesgue measure, and there exists some $\beta>0$ and $C\geq0$ such that:
\begin{equation}
\limsup_{r \rightarrow \infty} r^\beta \int_{\mathcal{X} \setminus B_r(\bm{y})} \mathbb{E}\|\bm{x} - \bm{Y}\|_2 \; dF(\bm{x}) \leq C \text{ for all $\bm{y} \in \mathcal{X}$},
\label{eq:moment}
\end{equation}
then $I(g;F,F_n) = \mathcal{O}\{\|g\|_{\mathcal{G}_p} {n}^{-1/2}(\log n)^{-(\gamma-\nu)/(2p)}\}$ for any $\nu \in (0, \gamma)$, where $\gamma = \beta/(\beta+1)$ and $B_r(\bm{y})$ denotes an $r$-ball around $\bm{y}$.
\end{enumerate}
Here, constant terms may depend on $\nu$, $p$ or $\beta$.
\label{thm:compact}
\end{theorem}

\begin{proof}
See Appendix A.2 of the supplemental article \citep{MJ2017supp}.
\end{proof}

In words, Theorem \ref{thm:compact} demonstrates the improvement of support points over MC under certain conditions on the sample space $\mathcal{X}$ or the desired distribution $F$. Specifically, part (a) requires the sample space $\mathcal{X}$ to be bounded with non-empty interior, whereas part (b) relaxes this boundedness restriction on $\mathcal{X}$ at the cost of the mild moment condition \eqref{eq:moment} on $F$. This condition holds for a large class of distributions which are not too heavy-tailed.

For illustration, consider the standard normal distribution for $F$, with sample space $\mathcal{X} = \mathbb{R}^p$. Note that, when $\|\bm{x}\|_2$ becomes large, $\mathbb{E}\|\bm{x} - \bm{Y}\|_2 \approx \|\bm{x}\|_2$. Hence, the condition in \eqref{eq:moment} becomes:
\[\limsup_{r \rightarrow \infty} r^\beta P(r), \quad P(r) \equiv (2\pi)^{-p/2} \int_{\mathbb{R}^p \setminus B_r(\bm{0})} \|\bm{x}\|_2 \exp\{-\|\bm{x}\|_2^2/2\}\; d\bm{x}.\]
Since $P'(r) \propto - r^p \exp\{-r^2/2\}$, it follows that $P(r) = \mathcal{O}(r^{p-1}\exp\{-r^2/2\})$, so $\limsup_{r\rightarrow \infty} r^\beta P(r) = 0$ for all $\beta > 0$. Applying part (b) of Theorem \ref{thm:compact}, support points enjoy a convergence rate of $\mathcal{O}\{n^{-1/2}(\log n)^{-(1-\nu)/(2p)}\}$ for any $\nu \in (0,1)$ in this case. An analogous argument shows a similar rate holds for any spherically symmetric distribution (see, e.g., \citep{FW1993}) with an exponentially decaying density in its radius.

\subsection{Comparison with MC and existing QMC methods}
\label{sec:conj}

We first discuss the implications of Theorems \ref{thm:convrate} and \ref{thm:compact} in comparison to Monte Carlo. Using the law of iterated logarithms \citep{Kie1961}, one can show that the error convergence rate for MC is bounded a.s. by $\mathcal{O}(n^{-1/2} \sqrt{\log \log n})$ for any distribution $F$. Comparing this with \eqref{eq:intconvrate}, the error rate of support points is asymptotically quicker than MC by at least some log-factor when dimension $p$ is fixed. This improvement is reflected in the simulations in Section \ref{sec:sim}, where support points enjoy a considerable improvement over MC for all point set sizes $n$. {When dimension $p$ is allowed to vary (and assuming $\|g\|_{\mathcal{G}_p}$ and $\text{Var}\{g(\bm{X})\}$, $\bm{X} \sim F$, do not depend on $p$), note that the MC rate is independent of $p$, while the rate in \eqref{eq:intconvrate} can have constants which depend on $p$.} From a theoretical perspective, this suggests support points may be inferior to MC for high-dimensional integration problems. Such a \textit{curse-of-dimensionality}, however, is not observed in our numerical experiments, where support points enjoy a sizable error reduction over MC for $p$ as large as 500.

Compared to existing QMC techniques, the existence rate in Theorem \ref{thm:convrate} falls short in the uniform setting of $F=U[0,1]^p$. For fixed dimension $p$, \citep{FW1993} showed that for any integrand $g$ with bounded variation (in the sense of Hardy and Krause), the error rate for classical QMC point sets is $\mathcal{O}\{n^{-1}(\log n)^p\}$, which is faster than \eqref{eq:intconvrate}. {Moreover, when $p$ is allowed to {vary}, it can be shown (see \citep{KS2005,Dea2013}) that certain {randomized} QMC (RQMC) methods, such as the {randomly-shifted lattice rules} in \citep{Sea2002}, enjoy a root-mean-squared error rate of $\mathcal{O}(n^{-1+\delta})$ with $\delta \in (0,1/2)$, where constant terms do not depend on dimension $p$.} On the other hand, support points provide \textit{optimal} integration points (in the sense of minimum energy) for \textit{non-uniform} distributions at fixed sample size $n$. Because of this optimality, support points can enjoy reduced errors to existing QMC methods in practice, which we demonstrate later for a specific RQMC method called randomly-scrambled Sobol' sequences \citep{Sob1967,Owe1998}. This suggests the rate in Theorem \ref{thm:convrate} may not be tight, and further theoretical work is needed (we outline one possible proof approach in Appendix A.3 of the supplemental article \citep{MJ2017supp}).

%

\section{Generating support points}
\label{sec:alg}
The primary appeal of support points is the efficiency by which these point sets can be optimized, made possible by exploiting the d.c. structure of the energy distance. Here, we present two algorithms, \texttt{sp.ccp} and \texttt{sp.sccp}, which employ a combination of the convex-concave procedure (CCP) with resampling to quickly optimize support points. \texttt{sp.ccp} should be used when sample batches are computationally expensive to obtain from $F$, whereas \texttt{sp.sccp} should be used when samples can be easily obtained. We prove the convergence of both algorithms to a stationary point set, and briefly discuss their running times.

\subsection{Algorithm statements}
We first present the steps for \texttt{sp.ccp}, then introduce \texttt{sp.sccp} as an improvement on \texttt{sp.ccp} when multiple sample batches from $F$ can be efficiently obtained. Suppose a single sample batch $\{\bm{y}_m\}_{m=1}^N$ is obtained from $F$. Using this, \texttt{sp.ccp} optimizes the following {Monte Carlo approximation} of the support points formulation \eqref{eq:spdef}:
\small
\begin{equation}
\argmin_{\bm{x}_1, \cdots, \bm{x}_n} \hat{{E}}(\{\bm{x}_i\};\{\bm{y}_m\}) \equiv \frac{2}{nN} \sum_{i=1}^n \sum_{m=1}^N  \|\bm{y}_m - \bm{x}_i \|_2 - \frac{1}{n^2} \sum_{i=1}^n \sum_{j=1}^n \|\bm{x}_i-\bm{x}_j\|_2.
\tag{MC}
\label{eq:spsamp}
\end{equation}
\normalsize
The approximated objective $\hat{{E}}$ was originally proposed by \citep{SR2004} as a two-sample GOF statistic for testing whether $\{\bm{y}_m\}_{m=1}^N$ and $\{\bm{x}_i\}_{i=1}^n$ are generated from the same distribution. Posed as an optimization problem, however, the goal in \eqref{eq:spsamp} is to {recover} the point set which best represents the random sample $\{\bm{y}_m\}_{m=1}^N$ from $F$ in terms of goodness-of-fit.

The key observation here is that the objective function $\hat{{E}}$ can be written as a difference of convex functions in $\bm{x} = (\bm{x}_1, \cdots, \bm{x}_n)$, namely, the two terms in \eqref{eq:spsamp}. This structure allows for efficient optimization using d.c. programming methods, which enjoy a well-established theoretical and numerical framework \citep{Tuy1995,TA1997}. While global optimization algorithms have been proposed for d.c. programs (e.g., \citep{Tuy1986}), such methods are typically quite slow in practice \citep{LB2014}, and may not be appropriate for the large-scale problem at hand. Instead, we employ a d.c. algorithm called the convex-concave procedure (CCP, see \citep{YR2003}) which, in conjunction with the {distance-based} property of the energy distance, allows for efficient optimization of \eqref{eq:spsamp}.

The main idea in CCP is to first replace the concave term in the d.c. objective with a convex upper bound, then solve the resulting ``surrogate'' formulation (which is convex) using convex programming techniques. This procedure is then repeated until the solution iterates converge. CCP can be seen as a specific case of majorization-minimization (MM, see \citep{Lan2016}), a popular optimization technique in statistics. The key to computational efficiency lies in finding a convex surrogate formulation which can be minimized in closed-form. Here, such a formulation can be obtained by exploiting the distance-based structure of \eqref{eq:spsamp}, with its closed-form minimizer given by the iterative map $\bm{x}_i^{[l+1]} \leftarrow M_i(\{\bm{x}_j^{[l]}\}_{j=1}^n; \{\bm{y}_m\}_{m=1}^N)$, $i=1, \cdots, n$, where $M_i$ is given in \eqref{eq:closedform}. The appeal of CCP here is two-fold. First, the evaluation of the iterative maps $M_i, i = 1, \cdots, n$ requires $\mathcal{O}(n^2p)$ work, thereby allowing for the efficient generation of moderately-sized point sets in moderately-high dimensions. Second, the computation of these maps can be greatly sped up using parallel computing, a point further discussed in Section \ref{sec:rt}.

Algorithm \ref{alg:ccp} outlines the detailed steps for \texttt{sp.ccp} following the above discussion. One caveat for \texttt{sp.ccp} is that it uses only {one} sample batch from $F$, even when multiple sample batches can be generated efficiently. This motivates the second algorithm, \texttt{sp.sccp}, whose steps are outlined in Algorithm \ref{alg:sccp}. The main difference for \texttt{sp.sccp} is that $\{\bm{y}_m\}_{m=1}^N$ is \textit{resampled} within each CCP iteration (a procedure known as {stochastic MM}). This resampling scheme allows \texttt{sp.sccp} to converge to a stationary point set for the {desired} problem \eqref{eq:spdef}, which we demonstrate next.

\begin{algorithm}[t]
\small
\caption{\texttt{sp.ccp}: Support points using one sample batch}
\label{alg:ccp}
\begin{algorithmic}
\stb Sample $\mathcal{D}^{[0]} = \{\bm{x}_i^{[0]}\}_{i=1}^n$ i.i.d. from $\{\bm{y}_m\}_{m=1}^N$.
\stb Set $l = 0$, and \textbf{repeat} until convergence of $\mathcal{D}^{[l]}$:
\bi[leftmargin=30pt]
\item \textbf{For} $i = 1, \cdots, n$ \textbf{do parallel}:
\bi[leftmargin=0pt]
\item Set $\bm{x}_i^{[l+1]} \leftarrow M_i(\mathcal{D}^{[l]}; \{\bm{y}_m\}_{m=1}^N)$, with $M_i$ defined in \eqref{eq:closedform}.
\ei
\item Update $\mathcal{D}^{[l+1]} \leftarrow \{\bm{x}_i^{[l+1]}\}_{i=1}^n$, and set $l \leftarrow l + 1$.
\ei
\stb Return the converged point set $\mathcal{D}^{[\infty]}$.
\end{algorithmic}
\normalsize
\end{algorithm}

\begin{algorithm}[t]
\small
\caption{\texttt{sp.sccp}: Support points using multiple sample batches}
\label{alg:sccp}
\begin{algorithmic}
\stb Sample $\mathcal{D}^{[0]} = \{\bm{x}_i^{[0]}\}_{i=1}^n \distas{i.i.d.} F$, \crd{set $(w_l)_{l=0}^\infty = (np/(np+l))_{l=0}^\infty$, $(\bar{d}_i^{[0]})_{i=1}^n = \bm{0}$}.
\stb Set $l = 0$, and \textbf{repeat} until convergence of $\mathcal{D}^{[l]}$:
\bi[leftmargin=30pt]
\item Resample $\crd{\mathcal{Y}^{[l]} = }\{\bm{y}^{[l]}_m\}_{m=1}^N \distas{i.i.d.} F$.
\item \textbf{For} $i = 1, \cdots, n$ \textbf{do parallel}:
\small
\bi[leftmargin=0pt]
\item Set $\bm{x}_i^{[l+1]} \leftarrow \crd{(1-\kappa_l) \bm{x}_i^{[l]} + \kappa_l} M_i(\mathcal{D}^{[l]}; \mathcal{Y}^{[l]})$, with $M_i$ in \eqref{eq:closedform}, where $\crd{\kappa_l = w_l q(\bm{x}_i^{[l]}; \mathcal{Y}^{[l]})/[{w_l q(\bm{x}_i^{[l]}; \mathcal{Y}^{[l]}) + (1-w_l) \bar{d}_i^{[l]}}]}$.
\item \crd{Set $\bar{d}_i^{[l+1]} \leftarrow (1-w_l) \bar{d}_i^{[l]} + w_l q(\bm{x}_i^{[l]}; \mathcal{Y}^{[l]})$}.
\ei
\normalsize
\item Update $\mathcal{D}^{[l+1]} \leftarrow \{\bm{x}_i^{[l+1]}\}_{i=1}^n$, and set $l \leftarrow l + 1$.
\ei
\stb Return the converged point set $\mathcal{D}^{[\infty]}$.
\end{algorithmic}
\normalsize
\end{algorithm}

\subsection{Algorithmic convergence}
For completeness, a brief overview of MM is provided, following \cite{Lan2016}.

\begin{definition}[Majorization function]
Let $f:\mathbb{R}^s \rightarrow \mathbb{R}$ be the objective function to be minimized. A function $h(\bm{z}|\bm{z}')$ \textup{majorizes} $f(\bm{z})$ at a point $\bm{z}' \in \mathbb{R}^s$ if $h(\bm{z}|\bm{z}') \geq f(\bm{z})$, with equality holding when $\bm{z} = \bm{z}'$.
\end{definition}
\noindent Starting at an initial point $\bm{z}^{[0]}$, the goal in MM is to minimize the majorizing function $h$ as a surrogate for the true objective $f$, and iterate the updates $\bm{z}^{[l+1]} \leftarrow \argmin_{\bm{z}} h(\bm{z}|\bm{z}^{[l]})$ until convergence. This iterative procedure has the so-called {descent property} $f(\bm{x}^{[l+1]}) \leq f(\bm{x}^{[l]})$, which ensures solution iterates are always decreasing in $f$. The key for efficiency is to find a majorizing function $g$ with a closed-form minimizer which is easy to compute.

Consider now the Monte Carlo approximation in \eqref{eq:spsamp}, which has a d.c. formulation in $\{\bm{x}_i\}_{i=1}^n$, with concave term $-n^{-2} \sum_{i=1}^n \sum_{j=1}^n \|\bm{x}_i - \bm{x}_j\|_2$. Following CCP, we first majorize this term using a first-order Taylor expansion at the current iterate $\{\bm{x}_j'\}_{j=1}^n$, yielding the surrogate convex program:
\begin{align}
\begin{split}
\small
&\argmin_{\bm{x}_1, \cdots, \bm{x}_n} h(\{\bm{x}_i\}_{i=1}^n; \{\bm{x}'_j\}_{j=1}^n) \\
& \equiv \frac{2}{nN} \sum_{i=1}^n \sum_{m=1}^N \|\bm{y}_m- \bm{x}_i\|_2 - \frac{1}{n^2} \left[ \sum_{i=1}^n \sum_{j=1}^n \left(\|\bm{x}_i' - \bm{x}_j'\|_2 + \frac{2(\bm{x}_i-\bm{x}_i')^T(\bm{x}_i' - \bm{x}_j')}{\|\bm{x}_i' - \bm{x}_j'\|_2} \right) \right].
\label{eq:convex}
\end{split}
\end{align}
Implicit here is the assumption that the current point set is pairwise distinct, i.e., $\bm{x}_i' \neq \bm{x}_j'$ for all $i,j = 1, \cdots, n$. From simulations, this appears to be always satisfied by initializing the algorithm with a pairwise distinct point set, because the random sampling of $\{\bm{y}_m\}$ and the ``almost-random'' round-off errors \citep{AS2011} in the evaluation of $M_i$ force subsequent point sets to be pairwise distinct. Such an assumption can also be easily checked after each iteration.

While \eqref{eq:convex} can be solved using gradient-based convex programming techniques, this can be computationally burdensome when $n$ or $p$ becomes large, because such methods may require many evaluations of $h$ and its subgradient. Instead, the following lemma allows us to perform a slight ``convexification'' of the convex term in \eqref{eq:convex}, which then yields a efficient closed-form minimizer.
\begin{lemma}[Convexification]
$Q(\bm{x}|\bm{x}') = \frac{\|\bm{x}\|_2^2}{2\|\bm{x}'\|_2} + \frac{\|\bm{x}'\|_2}{2}$ majorizes $\| \bm{x} \|_2$ at $\bm{x}'$ for any $\bm{x}' \in \mathbb{R}^p$.
\label{lem:maj}
\end{lemma}

\begin{proof}
See Appendix A.4 of the supplemental article \citep{MJ2017supp}.
\end{proof}

Lemma \ref{lem:maj} has an appealing geometric interpretation. Viewing $\|\bm{x}\|_2$ as a second-order cone centered at $\bm{0}$, $Q(\bm{x}|\bm{x}')$ can be interpreted as the tightest convex paraboloid intersecting this cone at $\bm{x}'$. Note that the quadratic nature of the majorizer $Q$, which is crucial for deriving a closed-form minimizer, is made possible by the distance-based structure of the energy distance.

From this, the following lemma provides a quadratic majorizer for \eqref{eq:convex}, along with its corresponding closed-form minimizer:
\begin{lemma}[Closed-form iterations]
Define the function $h^Q$ as:
\small
\begin{align*}
h^Q(\{\bm{x}_i\}_{i=1}^n; \{\bm{x}'_j\}_{j=1}^n) &\equiv \frac{2}{nN} \sum_{i=1}^n \sum_{m=1}^N \left\{ \frac{\|\bm{y}_m- \bm{x}_i\|_2^2}{2\|\bm{y}_m-\bm{x}_i'\|_2} + \frac{\|\bm{y}_m-\bm{x}_i'\|_2}{2} \right\} \\
& - \frac{1}{n^2} \left[ \sum_{i=1}^n \sum_{j=1}^n \left(\|\bm{x}_i' - \bm{x}_j'\|_2 + \frac{2(\bm{x}_i-\bm{x}_i')^T(\bm{x}_i' - \bm{x}_j')}{\|\bm{x}_i' - \bm{x}_j'\|_2} \right) \right],
\end{align*}
\normalsize
Then $h^Q(\cdot;\{\bm{x}_j'\}_{j=1}^n)$ majorizes $\hat{E}$ at $\{\bm{x}_j'\}_{j=1}^n$. Moreover, the global minimizer of $h^Q(\cdot;\{\bm{x}_j'\}_{j=1}^n)$ is given by:
\small
\begin{align}
\begin{split}
\bm{x}_i &= M_i(\{\bm{x}_j'\}_{j=1}^n; \{\bm{y}_m\}_{m=1}^N)\\
& \equiv \crd{q^{-1}(\bm{x}_i'; \{\bm{y}_m\}_{m=1}^N)} \left( \frac{N}{n} \sum_{\substack{j=1\\ j\neq i}}^n \frac{\bm{x}_i'-\bm{x}_j'}{\|\bm{x}_i'-\bm{x}_j'\|_2} + \sum_{m=1}^N \frac{\bm{y}_m}{\|\bm{x}_i'-\bm{y}_m\|_2} \right), \; i = 1, \cdots, n,
\label{eq:closedform}
\end{split}
\end{align}
\normalsize
where \crd{$q(\bm{x}_i; \{\bm{y}_m\}_{m=1}^N) \equiv \left(\sum_{m=1}^N \|\bm{x}_i - \bm{y}_m\|_2^{-1} \right)$}.
\label{lem:quad}
\end{lemma}

\begin{proof}
See Appendix A.5 of the supplemental article \citep{MJ2017supp}.
\end{proof}

One can now prove the convergence of \texttt{sp.ccp} and \texttt{sp.sccp}.

\begin{theorem} \emph{(Convergence - \texttt{sp.ccp})} Assume $\mathcal{X}$ is closed and convex. For any pairwise distinct $\mathcal{D}^{[0]} \subseteq \mathcal{X}$ and fixed sample batch $\{\bm{y}_m\}_{m=1}^ N \subseteq \mathcal{X}$, the sequence $(\mathcal{D}^{[l]})_{l=1}^\infty$ in Algorithm \ref{alg:ccp} converges to a limiting point set $\mathcal{D}^{[\infty]}$ which is stationary for $\hat{E}$.
\label{thm:conv1}
\end{theorem}
\begin{proof}
See Appendix A.6 of the supplemental article \citep{MJ2017supp}.
\end{proof}

\begin{theorem} \emph{(Convergence - \texttt{sp.sccp})} Assume $\mathcal{X}$ is compact and convex. For any pairwise distinct $\mathcal{D}^{[0]} \subseteq \mathcal{X}$, \crd{all limiting point sets $\mathcal{D}^{[\infty]}$ (there exists at least one) of the sequence $(\mathcal{D}^{[l]})_{l=1}^\infty$ in Algorithm \ref{alg:sccp} are stationary for $E$.}
\label{thm:conv2}
\end{theorem}
\begin{proof}
See Appendix A.7 of the supplemental article \citep{MJ2017supp}.
\end{proof}

\noindent (Recall that $\bm{z} \in D$ is a stationary solution for a function $f: D \subseteq \mathbb{R}^s \rightarrow \mathbb{R}$ if:
\[f'(\bm{z},\bm{d}) \geq 0 \quad \text{ for all } \bm{d} \in \mathbb{R}^s \; \text{s.t.} \; \bm{z} + \bm{d} \in D,\]
where $f'(\bm{z},\bm{d})$ is the directional derivative of $f$ at $\bm{z}$ in direction $\bm{d}$.) Note that the compactness condition on $\mathcal{X}$ in Theorem \ref{thm:conv2} is needed to prove the convergence of stochastic MM algorithms, since it allows for an application of the law of large numbers (see \cite{Mai2013} for details).


\subsection{Running time and parallelization}
\label{sec:rt}
Regarding the running time of \texttt{sp.ccp}, it is well known that MM algorithms enjoy a linear error convergence rate \citep{OR1970}. This means $L = \mathcal{O}(\log \delta^{-1})$ iterations of \eqref{eq:closedform} are sufficient for achieving an objective gap of $\delta > 0$ from the stationary solution. Since the maps in \eqref{eq:closedform} require $\mathcal{O}\{ n (n + N) p\}$ work to compute, the running time of \texttt{sp.ccp} is $\mathcal{O}\{n (n + N) p \log \delta^{-1}\}$. Assuming the batch sample size $N$ does not increase with $n$ or $p$, this time reduces to $\mathcal{O}(n^2p \log \delta^{-1})$, which suggests the proposed algorithm can efficiently generate moderately-sized point sets in moderately-high dimensions, but may be computationally burdensome for large point sets. While a similar linear error convergence is difficult to establish for \texttt{sp.sccp} due to its stochastic nature (see \citep{BB2008, GL2013}), its running time is quite similar to \texttt{sp.ccp} from simulations.

The separable form of \eqref{eq:closedform} also allows for further computational speed ups using parallel processing. As outlined in Algorithms \ref{alg:ccp} and \ref{alg:sccp}, the iterative map for each point $\bm{x}_i$ can be computed in parallel using separate processing cores. Letting $P$ be the total number of computation cores available, such a parallelization scheme reduces the running time of \texttt{sp.ccp} and \texttt{sp.sccp} to $\mathcal{O}( \lceil n/P \rceil n p \log \delta^{-1})$, thereby allowing for quicker optimization of large point sets. This feature is particularly valuable given the increasing availability of multi-core processors in personal laptops and computing clusters. 

\section{Simulations}
\label{sec:sim}
Several simulations are presented here which demonstrate the effectiveness of support points in practice. We first discuss the space-filling property of support points, then comment on its computation time using \texttt{sp.sccp}. Finally, we compare the integration performance of support points with MC and a RQMC method called IT-RSS (defined later).

\subsection{Visualization and timing}
\label{sec:vis}
\begin{figure}[t]
\centering
\includegraphics[width=0.9\textwidth]{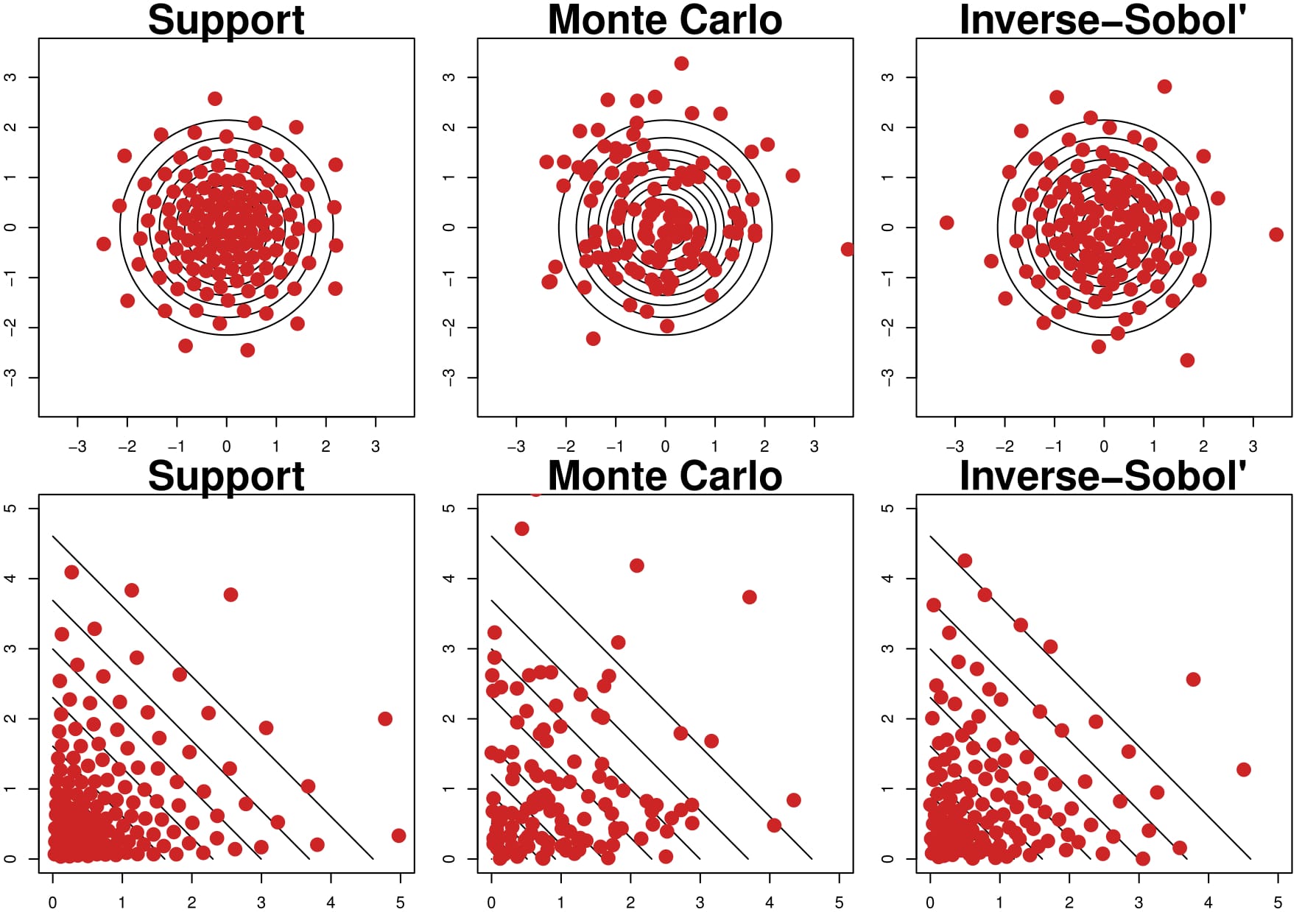}
\caption{{$n=128$} support points, MC points and inverse Sobol' points for i.i.d. $N(0,1)$ and $\text{Exp}(1)$ in $p=2$ dimensions. Lines represent density contours.}
\label{fig:visual}
\end{figure}

For visualization, Figure \ref{fig:visual} shows the $n=128$-point point sets for the i.i.d. $N(0,1)$ and $Exp(1)$ distributions in $p=2$ dimensions, with lines outlining density contours (additional visualizations provided in Appendix B of the supplemental article \citep{MJ2017supp}). Support points are plotted on the left, Monte Carlo samples in the middle and inverse Sobol' points on the right. {The latter is generated by choosing the Sobol' points on $U[0,1]^2$ which maximize the minimum interpoint distance over 10,000 random scramblings (see next section for details), then performing an inverse-transform of $F$ on such a point set.} From this figure, {support points appear to be slightly more visually representative of the underlying distribution $F$ than the inverse Sobol' points, and much more representative than MC.} Specifically, the proposed point set is concentrated in regions with high density, but each point is sufficiently spaced out from one another to maximize their representative power. Borrowing a term from design-of-experiments literature \cite{Sea2013}, we call point sets with these two properties to be \textit{space-filling} on $F$. A key reason for this space-fillingness is the {distance-based} property of the energy distance: the two terms for $E(F,F_n)$ in \eqref{eq:mvl2samp} force support points to not only mimic the desired distribution $F$, but also ensure no two points are too close together. This allows for a more appealing visual representation of $F$, and can provide more robust integration performance.

\begin{figure}[t]
\centering
\includegraphics[width=0.9\textwidth]{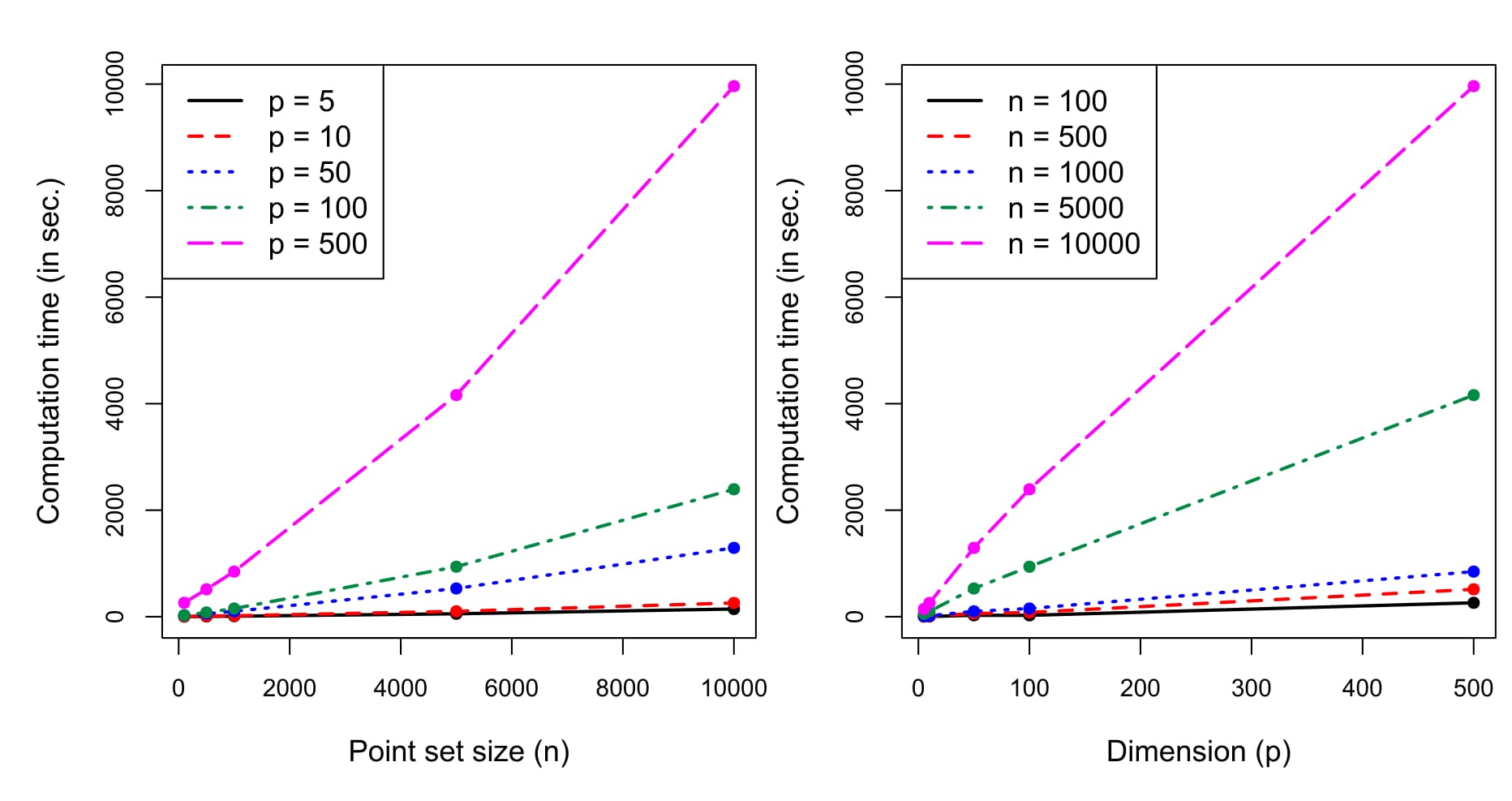}
\vspace{-0.4cm}
\caption{Computation time (in seconds) of \texttt{\textup{sp.sccp}} as a function of point set size ($n$) and dimension ($p$) for the i.i.d. $Beta(2,4)$ distribution.}
\label{fig:times}
\end{figure}

Regarding computation time, Figure \ref{fig:times} shows the times (in seconds) needed for \texttt{sp.sccp} to generate support points for the i.i.d. $Beta(2,4)$ distribution, first as a function of point set size $n$ with fixed dimension $p$, then as a function of $p$ with fixed $n$. The resampling size is fixed at $N = 10,000$ for all choices of $n$ and $p$. Similar times are reported for other distributions, and are not reported for brevity. All computations are performed on a 12-core Intel Xeon 3.50 Ghz processor. From this figure, two interesting observations can be made. First, for fixed $n$, these plots show that the empirical running times grow quite linearly in $p$, whereas for fixed $p$, these running times exhibit a slow quadratic (but almost linear) growth in $n$. This provides evidence for the $\mathcal{O}(n^2p)$ running time asserted in Section \ref{sec:rt}. Second, as a result of this running time, support points can be generated efficiently for moderate-sized point sets in moderately-high dimensions. For $p=2$, the required times for generating $n = 50 - 10,000$ points range from $3$ seconds to $2$ minutes; for $p=50$, $27$ seconds to $20$ minutes; and for $p=500$, $4$ minutes to $2.5$ hours. While these times are quite fast from an optimization perspective, they are still slower than number-theoretic QMC methods, which can generate, say, $n=10^6$ points in $p=10^3$ dimensions in a matter of seconds. {The appeal for support points is that, by exploiting the d.c. structure of the energy distance in \cite{SR2004}, one obtains for any distribution (locally) minimum energy sampling points which can outperform number-theoretic QMC methods.}

\subsection{Numerical integration}
\label{sec:integ}
We now investigate the integration performance of support points in comparison with Monte Carlo and an RQMC method called the inverse-transformed randomized Sobol' sequences (IT-RSS). {The former is implemented using the Mersenne twister \citep{MN1998}, the default pseudo-random number generator in the software R \citep{R2017}. The latter is obtained by (a) generating a randomized Sobol' sequence using the R package \texttt{randtoolbox} \citep{DS2013} (which employs Owen-style scrambling \citep{Owe1998} with Sobol' sequences generated in the implementation of \citep{JK2003}), and (b) performing the inverse-transform of $F$ on the resulting point set.} As mentioned in Section \ref{sec:sp}, IT-RSS performs well in the uniform setting $F=U[0,1]^p$, and provides a good benchmark for comparing support points with existing QMC methods.

The simulation set-up is as follows. Support points are generated using \texttt{sp.sccp}, with point set sizes ranging from $n=50$ to $10,000$ and resampling size $N$ fixed at $10,000$. Since MC and IT-RSS are randomized methods, we replicate both for 100 trials to provide an estimate of error variability, with replications seeded for reproducibility. Three distributions are considered for $F$: the i.i.d. $N(0,1)$, the i.i.d. $Exp(1)$ and the i.i.d. $Beta(2,4)$ distributions, with $p$ ranging from $5$ to $500$. For the integrand $g$, two (modified) test functions are taken from \cite{Gen1984}: the Gaussian peak function (GAPK): $g(\bm{x}) = \exp\left\{- \sum_{l=1}^p \alpha_l^2 (x_l - u_l)^2\right\}$ and the (modified) oscillatory function (OSC): $g(\bm{x}) = \exp\{-\sum_{l=1}^p \beta_l x_l^2 \}\cos\left(2\pi u_1 + \sum_{l=1}^p \beta_l x_l\right)$. Here, $\bm{x} = (x_l)_{l=1}^p$, $u_l$ is the marginal mean for the $l$-th dimension of $F$, and the scale parameters $\alpha_l$ and $\beta_l$ are set as $20/p$ and $5/p$, respectively.

\begin{figure}[t]
\centering
\includegraphics[width=0.9\textwidth]{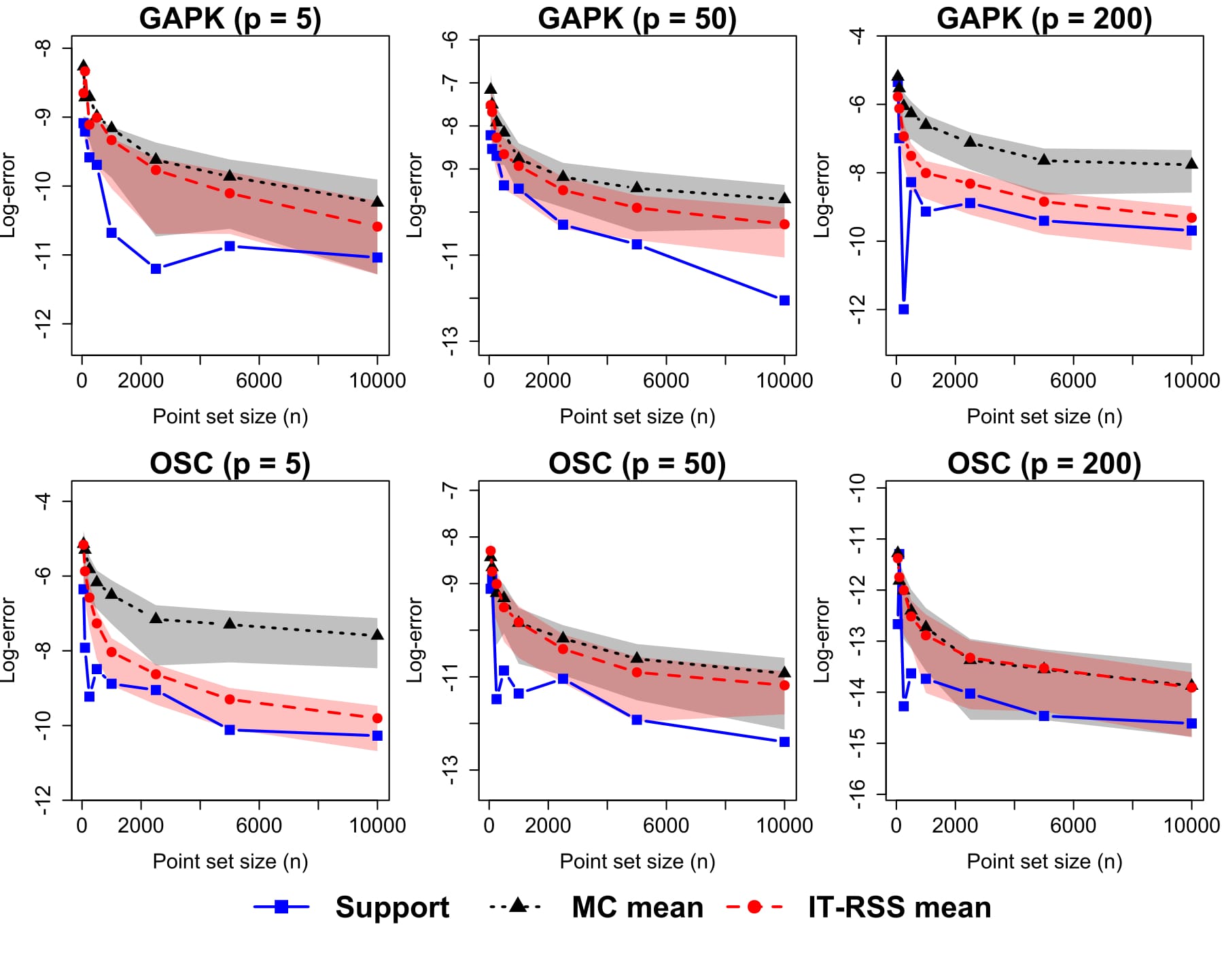}
\vspace{-0.5cm}
\caption{Log-absolute errors for GAPK under the i.i.d. $Exp(1)$ distribution (top) and for OSC under the i.i.d. $N(0,1)$ distribution (bottom). Lines denote log average-errors, and shaded bands mark the 25-th and 75-th quantiles.}
\label{fig:intmn}
\end{figure}

Figure \ref{fig:intmn} shows the resulting log-absolute errors in $p=5$, $50$ and $200$ dimensions for GAPK under the i.i.d. $Exp(1)$  distribution, and for OSC under the i.i.d. $N(0,1)$ distribution (results are similar for other settings, and are omitted for brevity). For MC and IT-RSS, the dotted lines indicate average error decay, and the shaded bands mark the area between the 25-th and 75-th error quantiles. Two observations can be made here. First, for all choices of $n$, support points enjoy considerably reduced errors compared to the averages of both MC and IT-RSS, with the proposed method providing an improvement to the 25-th quantiles of IT-RSS for most settings. Second, this advantage over MC and IT-RSS persists in both low and moderate dimensions. In view of the {relief from dimensionality} enjoyed by IT-RSS, this gives some evidence that support points \textit{may} enjoy a similar property as well, a stronger assertion than is provided in Theorem \ref{thm:convrate} or \ref{thm:compact}. Exploring the theoretical performance of support points in high dimensions will be an interesting direction for future work.

In summary, for point set sizes as large as $10,000$ points in dimensions as large as 500, simulations show that support points can be {efficiently generated} and enjoy {improved performance} over MC and IT-RSS. This opens up a wide range of important applications for support points in both small-data and big-data problems, two of which we describe next.

\section{Applications of support points}
\label{sec:app}

\subsection{Uncertainty propagation in expensive simulations}

\begin{figure}[t]
\centering
\includegraphics[width=0.9\textwidth]{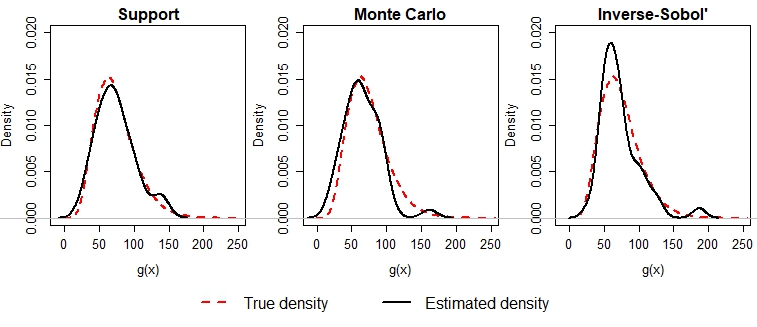}
\vspace{-0.4cm}
\caption{True and estimated density functions for $g(\bm{X})$ using $n=60$ points.}
\label{fig:boreholedens}
\end{figure}

We first highlight an important small-data application of support points in simulation. With the development of powerful computational tools, computer simulations are becoming the de-facto method for conducting engineering experiments. For such simulations, a key point of interest is \textit{uncertainty propagation}, or how uncertainty in input variables (resulting from, say, manufacturing tolerances) propagate and affect output variability. Mathematically, let $g(\bm{x})$ be the observed output at input setting $\bm{x}$, and let $\bm{X} \sim F$ denote input uncertainties. The distribution $g(\bm{X})$ can then be seen as the resulting uncertainty on system output. For engineers, the estimation of $g(\bm{X})$ using as few simulation runs as possible is of great importance, because each run can be computationally and monetarily expensive.

To demonstrate the effectiveness of support points for this problem, we use the \textit{borehole physical model} \citep{Wor1987}, which simulates water flow rate through a borehole. The 8 input variables for this model, along with their corresponding uncertainty distributions (assumed to be mutually independent), are summarized in Appendix C of the supplemental article \citep{MJ2017supp}. To reflect the expensive cost of simulations, we test only small point set sizes ranging from $n=20$ to $n=100$ runs. Support points are generated using \texttt{sp.sccp} with the same settings as before, with the randomized MC and IT-RSS methods replicated for 100 trials.

Consider now the estimation of the output distribution $g(\bm{X})$, which quantifies the uncertainty in water flow rate. Figure \ref{fig:boreholedens} compares the estimated density function of $g(\bm{X})$ using $n=60$ points with its true density, where the latter estimated using a large Monte Carlo sample. Visually, support points provide the best density approximation for $g(\bm{X})$, capturing well both the peak and tails of the desired output distribution. This suggests support points are not only asymptotically consistent for density estimation, but may also be optimal in some sense. A similar conclusion holds in the estimation of the expected flow rate $\mathbb{E}[g(\bm{X})]$ (see Appendix C of the supplemental article \citep{MJ2017supp}).


\subsection{Optimal MCMC reduction}

\begin{table}[t]
\centering
\begin{tabular}{c|c||c|c|c|c}
\toprule
\textit{Parameter} & \textit{Prior} & $R_{\mu}(375)$ &  $R_{\mu}(750)$ &  $R_{\sigma^2}(375)$ &  $R_{\sigma^2}(750)$ \\
\hline
$\phi_{i1}$ & $\log \phi_{i1} \distas{indep.} N(\mu_1,\sigma^2_1)$ & 2.27 & 2.75 & 15.89 & 6.37 \\
$\phi_{i2}$ & $\log (\phi_{i2}+1) \distas{indep.} N(\mu_2,\sigma^2_2)$ & 2.10 & 3.58 & 18.01 & 2.47 \\
$\phi_{i3}$ & $\log (-\phi_{i3}) \distas{indep.} N(\mu_3,\sigma^2_3)$ & 1.59 & 2.23& 11.90 & 102.49 \\
$\sigma^2_C$ & $ \sigma^2_C \sim \text{Inv-Gamma}(0.001,0.001)$ & 0.98 & 2.80 & 6.15 & 7.69 \\
\hline
$r(1600)$ & \multirow{3}{*}{\large{$r(t) = \frac{1}{5} \sum_{i=1}^5 \left. \frac{\partial}{\partial s} \eta_i(s) \right|_{s = t}$}} & 1.95 & 3.17 & - & - \\
$r(1625)$ & & 2.30 & 3.28 & - & - \\
$r(1650)$ & & 2.51 & 3.04 & - & - \\
\hline
$\mu_j$ & $\mu_j \distas{i.i.d.} N(0,100)$ & - & - & - & - \\
$\sigma^2_j$ & $\sigma^2_j \distas{i.i.d.} \text{Inv-Gamma}(0.01,0.01)$ & - & - & - & - \\
\toprule
\end{tabular}
\caption{Prior specification for the tree growth model (left), and the ratios of thinning over support point error for posterior quantities (right). $R_\mu(n)$ and $R_{\sigma^2}(n)$ denote the error ratios for posterior means and variances using $n$ points, respectively.}
\label{tbl:mcmc}
\end{table}

The second application of support points is as an improved alternative to MCMC thinning for Bayesian computation. {Thinning} here refers to the {discarding of all but every $k$-th sample} for an MCMC sample chain obtained from the posterior distribution. This is performed for several reasons (see \citep{LE2012}): it reduces high autocorrelations in the MCMC chain, saves computer storage space, and reduces processing time for computing derived posterior quantities. However, by carelessly throwing away samples, a glaring fault of thinning is that samples from thinned chains are inherently less accurate than that from the full chain. To this end, the proposed algorithm \texttt{sp.ccp} can provide considerable improvements to thinning by {optimizing} for a point set which best captures the distribution of the {full} MCMC chain.

We illustrate this improvement using the orange tree growth model in \citep{DS1981}. The data here consists of trunk circumference measurements ${\{Y_{i}(t_j)\}_{i=1}^5}_{j=1}^7$, where $Y_{i}(t_j)$ denotes the measurement taken on day $t_j$ from tree $i$. To model these measurements, the growth model $Y_{i}(t_j) \distas{indep.} N(\eta_{i}(t_j), \sigma^2_C), \eta_{i}(t_j) = {\phi_{i1}}/({1 + \phi_{i2} \exp\{ \phi_{i3} t_j \}})$ was assumed in \citep{DS1981}, where $\phi_{i1}$, $\phi_{i2}$ and $\phi_{i3}$ control the growth behavior of tree $i$. There are 16 parameters in total, which we denote by the set $\Theta = (\phi_{11}, \phi_{12}, \cdots, \phi_{53}, \sigma^2)$. Since no prior information is available on $\Theta$, vague priors are assigned, with the full specification provided in the left part of Table \ref{tbl:mcmc}. MCMC sampling is then performed for the posterior distribution using the R package \texttt{STAN} \citep{stan2015}, with the chain run for 150,000 iterations and the first 75,000 of these discarded as burn-in. The remaining $N = 75,000$ samples are then thinned at a rate of 200 and 100, giving $n = 375$ and $n=750$ thinned samples, respectively. Support points are generated using \texttt{sp.ccp} for the same choices of $n$, using the full MCMC chain as the approximating sample $\{\bm{y}_m\}_{m=1}^N$. Since posterior variances vary greatly between parameters, we first rescale each parameter in the MCMC chain to unit variance before performing \texttt{sp.ccp}, then scale back the resulting support points after.

These two methods are then compared on how well they estimate two quantities: (a) marginal posterior means and standard deviations of each parameter, and (b) the {averaged instantaneous growth rate} $r(t)$ (see Table \ref{tbl:mcmc}) at three future times. True posterior quantities are estimated by running a longer MCMC chain with 600,000 iterations. This comparison is summarized in the right part of Table \ref{tbl:mcmc}, which reports the ratios of thinning over support point error for each parameter. Keeping in mind that a ratio exceeding 1 indicates lower errors for support points, one can see that \texttt{sp.ccp} provides a sizable improvement over thinning for nearly all posterior quantities. Such a result should not be surprising, because \texttt{sp.ccp} compacts the \textit{full} MCMC chain into a set of optimal representative points, whereas thinning wastes valuable information by discarding a majority of this chain. 

\section{Conclusion and future work}
\label{sec:conc}

In this paper, a new method is proposed for compacting a continuous distribution $F$ into a set of representative points called support points, which are defined as the minimizer of the energy distance in \citep{SZ2013}. Three theorems are proven here which justify the use of these point sets for integration. First, we showed that support points are indeed representative of the desired distribution, in that these point sets converge in distribution to $F$. Second, we provided a Koksma-Hlawka-like bound which connects integration error with the energy distance for a large class of integrands. Lastly, using an existence result, we demonstrated the theoretical error improvement of support points over Monte Carlo. A key appeal of support points is its formulation as a difference-of-convex optimization problem. The two proposed algorithms, \texttt{sp.ccp} and \texttt{sp.sccp}, exploit this structure to efficiently generate moderate-sized point sets ($n \leq 10,000$) in moderately-high dimensions ($p \leq 500$). Simulations confirm the improved performance of support points to MC and a specific QMC method, and the practical applicability of the proposed point set is illustrated using two real-world applications, one for small-data and the other for big-data. An efficient C++ implementation of \texttt{sp.ccp} and \texttt{sp.sccp} is made available in the R package \texttt{support} \citep{Mak2017}.

While the current paper establishes some interesting results for support points, there are still many exciting avenues for future research. First, we are interested in exploring a tighter convergence rate for support points which reflects its empirical performance from simulations, particularly for high-dimensional problems. Next, the d.c. formulation of the energy distance can potentially be further exploited for the global optimization of support points. Moreover, by minimizing the \textit{distance-based} energy distance, support points also have an inherent link to the \textit{distance-based} designs used in computer experiments \citep{Sea2013,Jea2015b,MJ2017}, and exploring this connection may reveal interesting insights between the two fields, and open up new approaches for uncertainty quantification in engineering \citep{Mea2017} and machine-learning \citep{MX2017} problems. Lastly, motivated by \cite{Hic1998} and \cite{Jea2015b}, rep-points in high-dimensions should not only provide a good representation of the full distribution $F$, but also for \textit{marginal} distributions of $F$. Such a projective property is enjoyed by most QMC point sets in the literature \citep{Dea2013}, and new methodology is needed to incorporate this within the support points framework.

\section*{Acknowledgments}
The authors gratefully acknowledge helpful advice from an anonymous referee, the associate editor and Prof. Rui Tuo. This research is supported by the U. S. Army Research Office under grant number W911NF-14-1-0024.

\begin{supplement}
\sname{Supplement A}
\stitle{Additional proofs and results}
\slink[doi]{COMPLETED BY THE TYPESETTER}
\sdatatype{.pdf}
\sdescription{We provide in this supplement further details on technical results and simulation studies.}
\end{supplement}

\small
\bibliography{references}
\normalsize

\end{document}